\documentclass[a4wide,11pt]{article}

\usepackage{a4wide}
\usepackage[ansinew]{inputenc}
\usepackage{amsmath, amsfonts} 
\usepackage{amsthm, amssymb}
\usepackage[english]{babel}
\usepackage[T1]{fontenc}
\usepackage{lmodern}
\usepackage{graphicx}
\usepackage{subfigure}

\newtheorem{proposition}{Proposition}

\newtheorem{lemma}{Lemma}

\theoremstyle{definition}

\newtheorem{remark}{Remark}
\newtheorem{example}{Example}
\usepackage{verbatim}
\usepackage{array}
\usepackage{float}
\usepackage{subfigure}
\usepackage{color}
\usepackage{algorithm, algorithmicx, algpseudocode}

\newcommand{\guillemets}[1]{``#1''}

\newcommand{\mat}[2]{\left(\begin{array}{#1} #2 \end{array} \right)}
\newcommand{\tr}[1]{\mathrm{tr}{#1}}
\newcommand{\var}[1]{\mathrm{Var}{#1}}
\newcommand{\cov}[1]{\mathrm{Cov}{#1}}

\newcommand{\alex}[1]{{\color{black} #1}}

\begin{document}

\title{Spectral identification of networks using sparse measurements}
\author{A. Mauroy and J. Hendrickx}

\maketitle

\begin{abstract}
We propose a new method to recover global information about a network of interconnected dynamical systems based on observations made at a small number (possibly one) of its nodes. In contrast to classical identification of full graph topology, we focus on the identification of the spectral graph-theoretic properties of the network, a framework that we call \emph{spectral network identification}.

The main theoretical results connect the spectral properties of the network to the spectral properties of the dynamics, which are well-defined in the context of the so-called Koopman operator and can be extracted from data through the Dynamic Mode Decomposition algorithm. These results are obtained for networks of diffusively-coupled units that admit a stable equilibrium state. For large networks, a statistical approach is considered, which focuses on spectral moments of the network and is well-suited to the case of heterogeneous populations. 

Our framework provides efficient numerical methods to infer \emph{global} information on the network from sparse \emph{local} measurements at a few nodes. Numerical simulations show for instance the possibility of detecting  the mean number of connections or the addition of a new vertex using measurements made at one single node, that need not be representative of the other nodes' properties.
\end{abstract}

\section{Introduction}

A major problem in the context of complex networks of interacting dynamical systems, which has been considered for many years, is to predict the collective dynamics when the network topology is known. However, in many situations, it is often desirable to address the reverse problem of inferring the topology of the network from available data capturing the collective dynamics. This reverse problem is relevant in fields such as biology (e.g. reconstructing regulatory networks from gene expression data), neuroimaging (e.g. revealing the structural organization of the brain), and engineering (e.g. localizing failures in power grids or computer networks), to list a few. Network identification problems have received increasing attention over the past years, and the topic is actively growing in nonlinear systems theory. See e.g. the recent survey \cite{Timme_review}. Many methods have been developed, exploiting techniques from various fields: linearization \cite{Sauer_net_ident, Timme_identification}, velocities estimation \cite{Pikovsky_net_ident, Timme2_net_ident}, adaptive control \cite{Yu_net_estimation}, steady-state control \cite{Yu_Parlitz_steady_state}, optimization \cite{net_ident_opti}, compressed sensing \cite{Sauer_net_ident,net_ident_compressed_sens}, stochastic methods \cite{Ren_net_ident}, etc. These methods provide the structural (i.e. exact) connectivity of the underlying network and exploit to do so the dynamical nature of the individual units, which is often known, at least partially. In contrast, correlation-based methods using statistical measures \cite{network_ident_fMRI} or information-theoretic measures \cite{net_ident_info_theory} have also been developed, but they can only infer the effective (i.e. statistical) connectivity of the network.

Network identification methods developed in the framework of dynamical systems theory are usually not well-suited to the analysis of real networks such as biological networks, social networks, etc. Most of them are invasive, requiring the modification of the network connectivity or dynamics. In addition, some of them cannot be used \guillemets{offline} for data analysis, since they require to interact dynamically with the network. More importantly, all the methods proposed so far for full network reconstruction require measurements at all the nodes of the network. Partial measurements have been considered in \cite{Goncalves} in the context of linear time-invariant systems for a partial reconstruction of the network between the measured states, and yet the authors showed that the problem cannot be solved without additional information on the system. 
It can actually be shown that measuring all nodes is necessary for a full network reconstruction, and this is usually out of reach in large real networks.
Indeed, the number of sensors is limited and typically (much) smaller that the number of nodes. Some nodes of real networks might also not be accessible, or the only available information might be the averaged activity of a group of nodes lying in a given region of the network (e.g. electrical activity in a region of the brain). All these limitations motivate the network identification framework developed in this paper, which overcomes them.

In this work, we take the view that identifying the exact complete topology of large networks is not only practically impossible, as mentioned above, but also often unnecessary.  The presence or absence of an edge between two specific nodes is for instance often only marginally relevant when analyzing the global structure of a large network. 
For this reason, we focus instead on the identification of the spectral properties of the network, a framework that we call \emph{spectral network identification}. Note that the idea of estimating the spectral properties of networks has been considered in the control theory community (e.g. \cite{Banaszuk_eigenvalues, Franceschelli, Kibangou}), but in the case of specific (linear) consensus dynamics imposed at each node. On one hand, spectral properties do not reveal the exact full network topology (indeed, we cannot \guillemets{hear the shape} of a drum \cite{shape_drum}), so that the identification objective has been relaxed. On the other hand, they are a central theme of study in spectral graph theory \cite{Chung_book} ---where they are typically defined through the so-called Laplacian matrix associated with the network. They are shown to provide relevant information on the global network structure such as mean, minimum and maximum node degree, and connectivity, and they are reflected on the network dynamics, see e.g. \cite{Schaub}. For instance, the second smallest eigenvalue of the Laplacian matrix ---also called algebraic connectivity--- is related to the speed of information diffusion in the network (e.g. opinion propagation, spreading of epidemics) and plays a key role in studying network synchronization. More generally, the spectral properties of the network provide simple markers capturing the global network structure. These spectral markers can be used to detect a pathology or a fault and to compare different networks.

While classical full topology identification requires measurements at all the nodes of the network, we show that spectral network identification requires only sparse measurements in the network. This can be roughly explained by the fact that each node of a strongly connected network \guillemets{feels} the influence of all the other nodes. With the method developed in this paper, measurements can therefore be performed on a very small subset of nodes (e.g. only one in some cases) that might not be representative of the whole set of nodes. They can also be defined by a possibly nonlinear function of the states of several nodes, such as the average dynamics of a group of units. Moreover, the proposed method is not invasive and can be used offline.

Spectral properties of (nonlinear) dynamics are well-defined in a framework based on the so-called Koopman operator \cite{Mezic} and can be extracted from data through numerical methods such as the Dynamic Mode Decomposition (DMD) algorithm \cite{Schmid,Tu}. In this context, our main theoretical contribution is to connect these spectral properties of the collective network dynamics, which are measured, to the spectral properties of the network, which are to be inferred. These results are obtained in the case of a diffusive coupling for networks reaching a synchronized equilibrium, where the states of all units converge to the same value, a behavior which can be observed with excitable neurons, cardiac cells, opinion dynamics, and epidemics. For small networks, exact spectral identification is achieved. For large networks, a statistical approach is proposed, which focuses on spectral moments and is well-suited to the case of heterogeneous populations. \alex{With a few sparse measurements in the network, this framework allows} to estimate the average number of connections, detect the addition of a node to the network, and measure whether two units influence each other.

The rest of the paper is organized as follows. In Section \ref{sec:problem_stat}, the problem of spectral network identification is introduced. In Section \ref{sec:exact}, theoretical results are presented in the case of linear and nonlinear networks. At the end of the section, we also discuss the limitations encountered for the exact spectral identification of non-identical units. Section \ref{sec:stat} focuses on large networks and provides statistical results related to the spectral moments of the Laplacian matrix, which are related to the statistical distribution of the node degrees. Numerical aspects of the methods are discussed in Section \ref{sec:numerical} and illustrated with several applications in Section \ref{sec:examples}. Concluding remarks and perspectives are given in Section \ref{sec:conclu}.

\section{Problem statement}
\label{sec:problem_stat}

\subsection{Classical vs spectral network identification}

A networked dynamical system consists of a set of $n$ interconnected dynamical systems (or units) interacting on a (weighted) graph $\mathcal{G}=(V,E,w)$. The system is deterministic since no stochastic perturbation is considered. Each unit is attached to a node (or vertex) $k \in V$ and is directly influenced by another unit $i$ if $(i,k)$ is an edge of the graph, i.e. $(i,k) \in E \subseteq V \times V$. The strength of the interaction between units is determined by the function $w:E \to \mathbb{R}^+$ which assigns a weight to each edge. The (weighted) degree $d_i$ of a vertex $i$ is given by
\begin{equation*}
d_i = \sum_{\substack{k=1}}^n w(i,k)
\end{equation*}
where, with a slight abuse of notation, $w(i,k)=w((i,k))$ if $(i,k) \in E$ and $w(i,k)=0$ if $(i,k) \notin E$. We assume $(i,i) \notin E$ for all $i \in V$. The graph $\mathcal{G}$ is represented by its (weighted) adjacency matrix $W$ defined with the entries
\begin{equation*}
W_{ij} = w(i,j)\,.
\end{equation*}
Alternatively, the graph is also described by the (weighted) Laplacian matrix
\begin{equation*}
L=D-W\,,
\end{equation*}
with the degree matrix $D=\mathrm{diag}(d_1 \, \cdots \, d_n)$. In the following, we consider graphs that can be weighted and directed, i.e. $(i,j) \in E$ does not imply $(j,i) \in E$, so that $W \neq W^T$ and $L \neq L^T$ in general.\\

The networked dynamical system is completely defined by
\begin{enumerate}
\item[(i)] the graph $\mathcal{G}$;
\item[(ii)] the local dynamics of the states $x_k \in \mathbb{R}^m$, $k\in V$, of the units attached to the vertices;
\item[(iii)] the type of coupling between pairs of interacting units.
\end{enumerate}
Here we are interested in the following network identification problem: under the assumption that (ii) and (iii) are known, infer the graph topology (i) from measurements of the state of the units. In particular, classical and spectral network identification problems are defined as follows.
\begin{itemize}
\item[A.] \textbf{Classical network identification.} Suppose that (ii) and (iii) are known. From measurements of the states of \emph{all the units} of the network, infer the set of edges $E$ and the weight function $w$.
\item[B.] \textbf{Spectral network identification.} Suppose that (ii) and (iii) are known. From $p \ll n$ measurements of the states of \emph{a small subset $\bar{V} \subset V$ of units}, estimate the spectrum $\sigma(L)$ of the Laplacian matrix $L$ (i.e. the Laplacian eigenvalues) or the first spectral moments $\mathcal{M}_k(L)=\frac{1}{n}\sum_{\lambda \in \sigma(L)} \lambda^k$ in the case of large networks (see Section \ref{sec:spectral_graph_th}).
\end{itemize}
In this paper, we will develop the framework of spectral network identification. We focus on networked systems that admit a stable equilibrium corresponding to the synchronization of the units. We also make the standing assumption that the units interact through a diffusive coupling. 

\subsection{What does spectral information reveal about the graph?}
\label{sec:spectral_graph_th}

In contrast to classical network identification, spectral network identification only requires sparse measurements in the network. The price to pay is the relaxed objective of getting only the Laplacian eigenvalues. Although this spectral information does not reveal the complete graph structure, it captures important topological properties of the network. We will not review the vast literature related to spectral graph theory (we refer the interested reader to \cite{Chung_book}), but provide here some basic results connecting the Laplacian spectrum to the topological properties of the graph.

In the case of a connected graph, the second smallest eigenvalue $\lambda_2$---called algebraic connectivity---captures the connectivity of the graph; see also the Cheeger inequality in undirected graphs. The algebraic connectivity is related to the time constant of the dominant dynamics and to the speed of information propagation in the network. It also provides a bound on the diameter of the graph, i.e. the longest path between any pair of vertices \cite{Mohar}. Moreover, the algebraic connectivity $\lambda_2$ and the spectral radius $\lambda_n$ (i.e. the largest eigenvalue) can be used to derive bounds on the minimal and maximal vertex degrees $d_{min}$ and $d_{max}$. In particular, for an undirected graph, we have \cite{Fiedler}
\begin{equation}
\label{d_min_max}
d_{min} \geq \frac{n-1}{n} \lambda_2 \qquad d_{max} \leq \frac{n-1}{n} \lambda_n \,.
\end{equation}

In the case of large graphs, it is convenient to consider the spectral moments 
\begin{equation}
\label{spec_moments}
\mathcal{M}_k(L)=\frac{1}{n}\sum_{\lambda \in \sigma(L)} \lambda^k = \frac{1}{n} \tr (L^k) \quad k \in \mathbb{N} \\
\end{equation}
which are related to the moments of the degree distribution. The first spectral moment is equal to the mean vertex degree $\mathcal{D}_1(\mathcal{G}) \triangleq \frac{1}{n} \sum_i d_i$, i.e.
\begin{equation}
\label{mean_deg}
\mathcal{M}_1(L)= \mathcal{D}_1(\mathcal{G})\,,
\end{equation}
and the first and second spectral moments give bounds on the quadratic mean of the degree distribution $\mathcal{D}_2(\mathcal{G}) \triangleq \frac{1}{n} \sum_i d_i^2$:
\begin{equation}
\label{quad_mean_deg}
\max\left((\mathcal{M}_1(L))^2,\frac{\mathcal{M}_2(L)}{2}\right) \leq  \mathcal{D}_2(\mathcal{G}) \leq \mathcal{M}_2(L)\,.
\end{equation}
The equality \eqref{mean_deg} is trivial and a short proof of the inequalities \eqref{quad_mean_deg} is given in Appendix \ref{sec:app_proofs}. In the case of undirected and unweighted graphs, other relationships can be derived, which link the spectral moments of $L$ to local structural features of the network \cite{Preciado}.

\section{Exact spectral identification}
\label{sec:exact}

In this section, we develop the spectral network identification framework in the case of identical units. We first consider linear systems and then extend the results to nonlinear dynamics. The main results of this section provide the exact connection between the spectral properties of the collective dynamics and the spectral properties of the network.

\subsection{Linear systems with identical units}
\label{sec:lin_ident}

Consider a network of $n$ identical units that are each described by $m$ states evolving according to the linear dynamics
\begin{equation}
\label{local_dyn}
\begin{array}{rclc}
\dot{x}_k & = & A x_k + B u_k & \qquad x_k \in \mathbb{R}^m \\
y_k & = & C^T x_k & \qquad y_k \in \mathbb{R}
\end{array}
\qquad k \in \{1,\dots,n\}
\end{equation}
with $A\in \mathbb{R}^{m \times m}$, $B \in \mathbb{R}^{m \times 1}$, and $C \in \mathbb{R}^{m \times 1}$ (which are assumed to be known). The interaction between the units is given by the diffusive coupling
\begin{equation}
\label{diff_coupling}
u_k=\sum_{j=1}^n W_{kj} (y_k-y_j) \,.
\end{equation}
Considering the state vector $X=[x_1 \dots x_n]^T \in \mathbb{R}^{mn}$, we have
\begin{equation*}
\dot{X}=(I_n \otimes A - L \otimes BC^T)X
\end{equation*}
where $I_n$ is the $n \times n$ identity matrix and $L \in \mathbb{R}^{n \times n}$ is the Laplacian matrix. We denote
\begin{equation}
\label{K}
K \triangleq I_n \otimes A - L \otimes BC^T \in \mathbb{R}^{mn \times mn}
\end{equation}
and the solution of $\dot{X}=K X$ is given by
\begin{equation*}
X(t)=\sum_{j=1}^{mn} V_j \, e^{\mu_j t}
\end{equation*}
where $V_j$ is an eigenvector of $K$ and $\mu_j$ is the corresponding eigenvalue. Note that $V_j$ depends on the initial condition $X(0)$. We assume that $A$, $B$, and $C$ are so that the units synchronize, i.e. $\lim_{t\rightarrow \infty} X(t)=[0 \dots 0]^T$, or equivalently $\Re\{\mu_j\} < 0$ for all $j$.

In the context of spectral network identification, measurements are performed through the linear observation function $f(X)=Q^T X \in \mathbb{R}^p$, where $Q \in \mathbb{R}^{mn \times p}$ is a sparse matrix and $p\ll n$ is the number of measurements. (Note that the case of a nonlinear observation function will be treated together with the case of nonlinear dynamics in Section \ref{sec:nonlinear}.) It is clear that all eigenvalues $\mu_j$ of $K$ appear in the expression of the measurement
\begin{equation}
\label{state_evol}
f(X(t))=\sum_{j=1}^{mn} Q^T V_j \, e^{\mu_j t}
\end{equation}
(provided that $f(V_j)=Q^T V_j\neq 0$ for all $j$ \footnote{This condition is equivalent to the observability of the pair $(K,Q)$, i.e. $\textrm{rank}([Q \,\, K^TQ \, \cdots \, (AK^T)^{nm-1}Q])=nm$.}) and this is true even if only one state of one vertex is measured, i.e. $f(X)=(e^k)^T X=X_k \in \mathbb{R}$, where $e^k$ is the $k$th unit vector. Therefore, estimations $\tilde{\mu}_j$ of the eigenvalues $\mu_j$ can be computed from snapshots of \eqref{state_evol}. To do so, one can use the so-called Dynamic Mode Decomposition (DMD) algorithm \cite{Schmid,Tu}. The algorithm is described in detail in Appendix \ref{app_algo}, and its numerical implementation is discussed in Section \ref{sec:numerical}. The efficiency and accuracy of the algorithm will be illustrated in the sequel through several examples.

\begin{remark} The DMD algorithm works in practice with a set of time series obtained with several initial conditions. We therefore assume that measurements are performed while the network is reset several times to a state different from its equilibrium point. Accurate results can also be obtained when only the states of a group of vertices are reset, provided that this group is large enough. In the sequel, we will however consider that all the states of the network are reset (as it happens for instance in real networks of excitable neurons or cardiac cells). Note also that Section \ref{sec:numerical} provides a way to decrease the number of time series used with the DMD algorithm. \hfill $\diamond$
\end{remark}

What remains to show is that the spectrum $\sigma(L)$ of $L$ can be inferred from the (measured) spectrum $\sigma(K)$ of $K$ when the local dynamics (i.e. $A$, $B$, and $C$) are known. The following lemma provides a relationship between $\sigma(K)$ and $\sigma(L)$.
\begin{lemma}
\label{lem_main}
For $K=I_n \otimes A - L \otimes BC^T$, we have
\begin{equation}
\label{spec_K}
\sigma(K)=\bigcup_{\lambda \in \sigma(L)} \sigma(A-\lambda BC^T) \,.
\end{equation}
The $mn$ eigenvectors of $K$ are given by 
\begin{equation}
\label{eigenvector_K}
V=v \otimes w\,,
\end{equation}
with $L v=\lambda v$ and $(A-\lambda BC^T) w = \mu w$.
\end{lemma}
\begin{proof}
We have
\begin{equation*}
\begin{split}
K V & = (I_n \otimes A - L \otimes BC^T)(v \otimes w) \\
& = v \otimes A w - \lambda v \otimes BC^T w \\
& = v \otimes (A- \lambda BC^T) w \\
& = \mu \, (v \otimes w) \\
& = \mu \, V \,,
\end{split}
\end{equation*}
so that any vector $V$ of the form \eqref{eigenvector_K} is an eigenvector of $K$ associated with the eigenvalue $\mu$. We have to show that $K$ does not admit other eigenvalues associated with other eigenvectors. If the matrices $A-\lambda BC^T$ have $m$ independent eigenvectors $w$, then a complete set of $n m$ independent eigenvectors is given by \eqref{eigenvector_K}. If a matrix $A-\lambda BC^T$ has less than $m$ independent eigenvectors, then it admits an eigenvalue with (algebraic) multiplicity $\alpha>1$ and
\begin{equation*}
(A-\lambda BC^T-\mu I_m)^\alpha \, \tilde{w} =0 \qquad (A-\lambda BC^T-\mu I_m)^{\alpha-1} \, \tilde{w} \neq 0
\end{equation*}
where $\tilde{w}$ is a generalized eigenvector of $A-\lambda BC^T$. In this case, we have
\begin{equation*}
(I \otimes A - L \otimes BC^T - \mu I_n \otimes I_m )^\alpha \, (v \otimes \tilde{w}) = v \otimes (A-\lambda BC^T-\mu I_m)^\alpha \, \tilde{w} =0
\end{equation*}
and
\begin{equation*}
(I \otimes A - L \otimes BC^T - \mu I_n \otimes I_m )^{\alpha-1} \, (v \otimes \tilde{w}) = v \otimes (A-\lambda BC^T-\mu I_m)^{\alpha-1} \, \tilde{w} \neq 0\,,
\end{equation*}
so that $v \otimes \tilde{w}$ is a generalized eigenvector of $K$ and $\mu$ is also of multiplicity $\alpha$. It follows that there is no other eigenvector than \eqref{eigenvector_K}. This concludes the proof.
\end{proof}

We remark that the result implies that $\sigma(A) \subset \sigma(K)$, since $0 \in \sigma(L)$. Now we can show that the spectral identification problem is consistent: there exists a bijection between the two spectra $\sigma(L)$ and $\sigma(K)$ (for fixed $A$, $B$, and $C$), so that $\sigma(L)$ can be inferred from $\sigma(K)$.
\begin{proposition}
\label{prop:net_ident}
Assume that the local dynamics \eqref{local_dyn} is controllable and observable (i.e. \\
$\textrm{rank}([B \, AB \, \cdots \, A^{m-1}B])=m$ and $\textrm{rank}([C \, A^TC \, \cdots \, (A^T)^{m-1}C])=m$, respectively). Then,
\begin{equation}
\label{sigma_L}
\sigma(L)=\{g(\mu)|\mu \in \sigma(K)\}
\end{equation}
with
\begin{equation}
\label{g}
g(\mu) = \begin{cases}
1/(C^T(A-\mu I_m)^{-1}B) & \textrm{if } \mu \notin \sigma(A) \,, \\
0 & \textrm{if } \mu \in \sigma(A) \,.
\end{cases}
\end{equation}
Moreover,
\begin{equation}
\label{rel_K_L}
\sigma(K_1)=\sigma(K_2) \Leftrightarrow \sigma(L_1)=\sigma(L_2)
\end{equation}
with $K_1=I_n \otimes A - L_1 \otimes BC^T$ and $K_2=I_n \otimes A - L_2 \otimes BC^T$.
\end{proposition}
\begin{proof}
We first show that $g(\mu) \in \sigma(L)$ for all $\mu \in \sigma(K)$. For $\mu \in \sigma(A)$, it is clear that $g(\mu)=0 \in \sigma(L)$. In the case $\mu \notin \sigma(A)$, it follows from Lemma \ref{lem_main} that $\mu \in \sigma(A- \lambda BC^T)$
for some $\lambda \in \sigma(L)$. This implies that there exists $w\in \mathbb{R}^m$ such that, for some $\lambda \in \sigma(L)$,
\begin{equation}
\label{gen_eigen}
(A-\mu I_m) w = \lambda B C^T w \,.
\end{equation}
It is clear from \eqref{gen_eigen} that $C^T w \neq 0$ (since $\mu \notin \sigma(A)$), so that we can assume $C^T w = 1$ without loss of generality. It follows that $w=\lambda (A-\mu I_m)^{-1} B$ and premultiplying by $C^T$, we obtain
the unique solution of \eqref{gen_eigen}
\begin{equation}
\label{g_expression}
\lambda = \frac{1}{C^T(A-\mu I_m)^{-1}B} = g(\mu)\,.
\end{equation}
Since $\lambda \in \sigma(L)$, we have $g(\mu) \in \sigma(L)$.

Now we show that there does not exist $\lambda \in \sigma(L)$ such that $\lambda \neq g(\mu)$ for all $\mu \in \sigma(K)$. Assume such $\lambda$ exists. If there exists $\mu \in \sigma(A-\lambda BC^T)$ that satisfies $\mu \notin \sigma(A)$, then \eqref{g_expression} is the unique solution of \eqref{gen_eigen}. We have $\lambda=g(\mu)$ and Lemma \ref{lem_main} implies $\mu \in \sigma(K)$. This is a contradiction. If all $\mu \in \sigma(A-\lambda BC^T)$ satisfy $\mu \in \sigma(A)$, then either
\begin{itemize}
\item[(a)] $\lambda=0$;
\item[(b)] $\lambda \neq 0$ if $C^Tw \neq 0$ and $B$ is in the image of $(A-\mu I_m)$, i.e. if $B$ is in the span of $m-1$ right eigenvectors of $A$;
\item[(c)] $\lambda \neq 0$ if $C^Tw = 0$ and $w$ is the right eigenvector of $A$ associated with the eigenvalue $\mu$, i.e. $C$ is in the span of $m-1$ left eigenvectors of $A$.
\end{itemize}
Since $(A,B)$ is a controllable pair, $B$ cannot be in the span of $m-1$ right eigenvectors of $A$. Since $(A,C)$ is an observable pair, $C$ cannot be in the span of $m-1$ left eigenvectors of $A$. It follows that the cases (b) and (c) are impossible, so that we have $\lambda=0=g(\mu)$, which is a contradiction. This concludes the first part of the proof.

Finally, it is clear from \eqref{spec_K} that $\sigma(L_1)=\sigma(L_2) \Rightarrow \sigma(K_1)=\sigma(K_2)$ and from \eqref{sigma_L} that $\sigma(K_1)=\sigma(K_2) \Rightarrow \sigma(L_1)=\sigma(L_2)$.
\end{proof}

\begin{remark}
When the local dynamics of the units is not completely controllable or observable, it is still possible to infer the spectrum of $L$ from the spectrum of $K$. When $B$ is in the span of the right eigenvectors $\{v_1,\dots,v_k\}$ of $A$ (or when $C$ is in the span of the left eigenvectors $\{\tilde{v}_1,\dots,\tilde{v}_k\}$ of $A$), it is easy to show that
\begin{equation*}
\sigma(A-\lambda BC^T) =\sigma(A) \setminus \{\lambda^A_1,\dots,\lambda^A_k\} \cup \{\mu \in \mathbb{C}|g(\mu)=\lambda\} 
\end{equation*}
with $A v_k = \lambda^A_k v_k$, so that
\begin{equation*}
\sigma(K) = \sigma(A) \cup \{\mu  \in \mathbb{C} | g(\mu)=\lambda, \lambda \in \sigma(L)\} \,.
\end{equation*}
For instance, if $B=v_1$ (or $C=\tilde{v}_1$), we have
\begin{equation*}
\sigma(K) = \sigma(A) \cup \{\lambda^A_1 - C^TB \lambda|\lambda \in \sigma(L)\} \,.
\end{equation*}
Proposition \ref{prop:net_ident} does not hold, since different spectra of $L$ can be associated with the same spectrum $\sigma(K)=\sigma(A)$. For instance, the spectra $\sigma(L)=\{0,(\lambda^A_2-\lambda^A_1)/(C^TB)\}$ and $\sigma(L)=\{0,(\lambda^A_3-\lambda^A_1)/(C^TB)\}$ are associated with the same spectrum $\sigma(K)=\sigma(A)$. Note however that \eqref{rel_K_L} still holds if one takes into account the multiplicity of the eigenvalues (provided that $C^TB \neq 0$). Moreover \eqref{g} can still be used to obtain the spectrum of $L$ from the spectrum of $K$. \hfill $\diamond$
\end{remark}

The spectral identification method is illustrated in the following simple example.

\begin{example}[Linear system] \label{example1} Consider a random network with $10$ vertices (see the adjacency matrix in Appendix \ref{app_examples}) with the local linear dynamics
\begin{equation}
\label{lin_dyn_example1}
\begin{array}{rcl}
\dot{x}_k & = & \mat{cc}
{  -1  &  -2 \\
     1  &  -1}
 x_k + \mat{c}{1 \\ 2} u_k \,, \\
y_k & = & (1 \quad 1) \  x_k \,,
\end{array}
\end{equation}
and assume that the observation function is $f(X)=x_{1,1}$ with $x_k=(x_{k,1},x_{k,2})$ (i.e. only one state of one unit is measured). Using the time series related to $10$ different initial conditions, the DMD algorithm provides an accurate estimate of the $mn=20$ eigenvalues of the matrix $K$ (Figure \ref{fig:example1}(a)). Then the Laplacian spectrum of the network is also recovered by using \eqref{g} (Figure \ref{fig:example1}(b)). \hfill $\diamond$

\begin{figure}[htbp]
\centering
\subfigure[Spectrum of $K$]{\includegraphics[width=0.45\columnwidth]{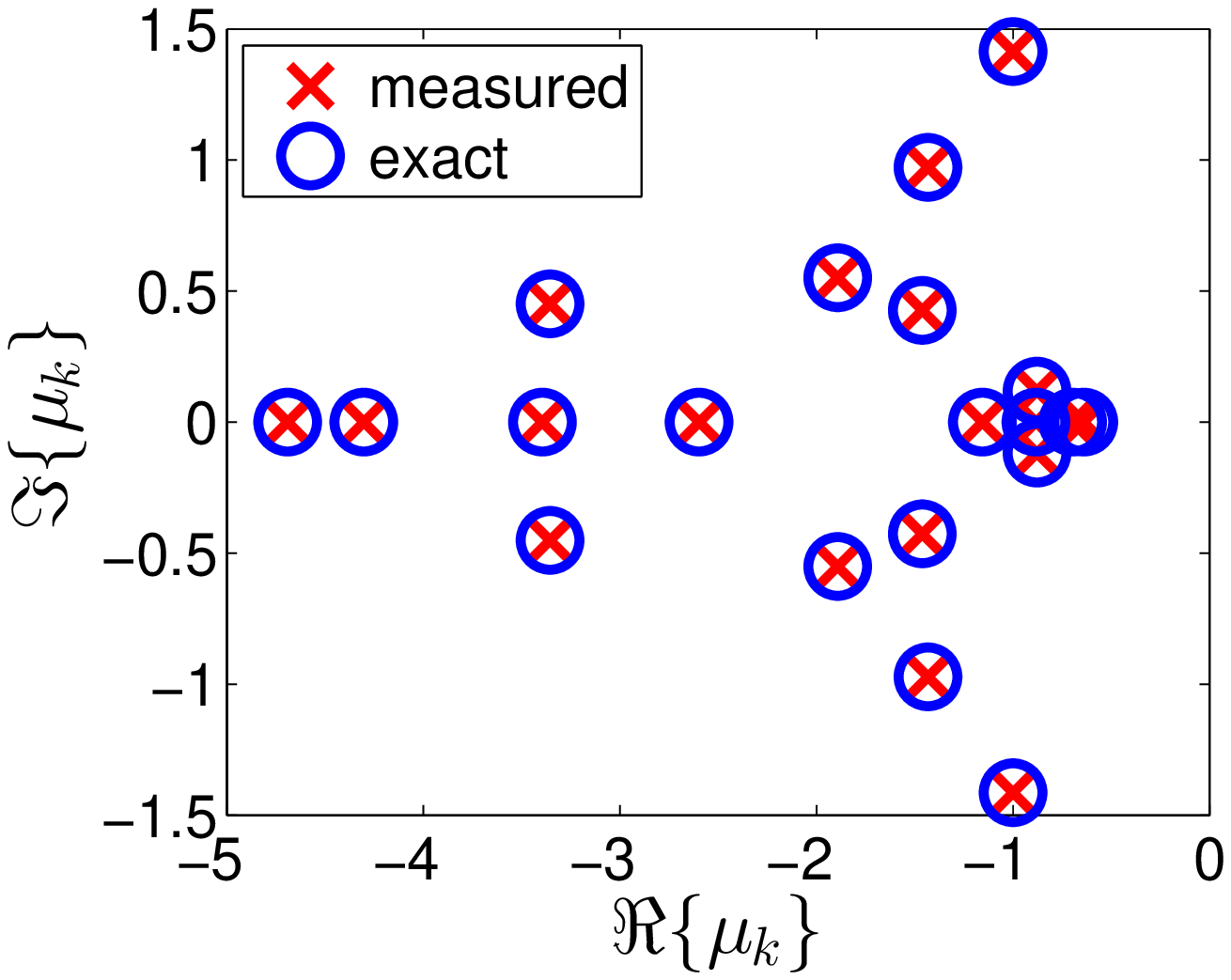}}
\subfigure[Spectrum of $L$]{\includegraphics[width=0.45\columnwidth]{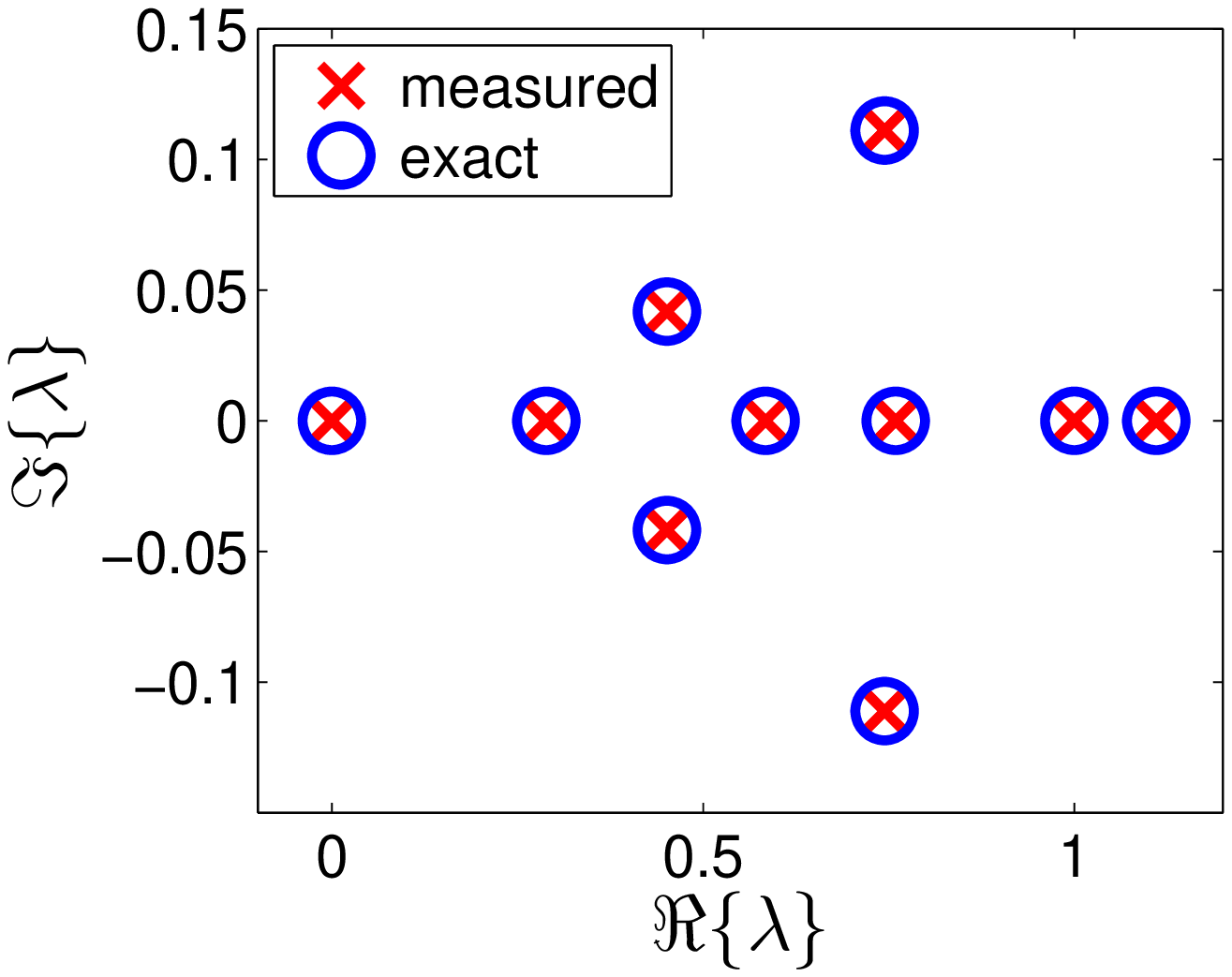}}
\caption{Spectral network identification for a linear system, using the measurement of one state at one vertex (Example \ref{example1}). Simulation parameters are given in Appendix \ref{app_examples}. (a) The DMD algorithm yields the 20 eigenvalues $\mu_k$ of the matrix $K$. (b) All the Laplacian eigenvalues $\lambda$ of the network are recovered.}
\label{fig:example1}
\end{figure}

\end{example}

\paragraph{Numerical error} Due to numerical imprecision, the eigenvalues of $K$ might be computed with some error (see e.g. Example \ref{example2}). Moreover these eigenvalues cannot be obtained precisely in the case of heterogeneous populations of non-identical units, as we will see. In these situations, one can estimate the induced error on the eigenvalues of $L$ obtained with \eqref{g_expression}. Denoting by $\tilde{\mu}=\mu+\delta \mu$ and $\tilde{\lambda}=\lambda+\delta \lambda$ a perturbed eigenvalue of $K$ and $L$, respectively, we have the first order Taylor approximation
\begin{equation*}
(A- (\mu + \delta \mu) I_m)^{-1} = (A-\mu I_M)^{-1} + (A-\mu I)^{-2} \delta \mu + \mathcal{O}(\delta \mu^2)
\end{equation*}
and \eqref{g_expression} yields
\begin{equation}
\label{sensitivity}
\delta \lambda \approx -\frac{C^T(A-\mu I_m)^{-2} B}{(C^T(A-\mu I_m)^{-1}B)^2} \delta \mu \triangleq \Delta(\mu) \delta \mu \,.
\end{equation}
It follows that measured eigenvalues $\tilde{\mu}$ of $K$ satisfying $|C^T(A-\mu I_m)^{-1} B| \ll 1$ will induce a large error on the associated eigenvalue $\tilde{\lambda}$ of $L$. Since $m$ distinct measured eigenvalues $\tilde{\mu}_k$, $k=1,\dots,m$, yield eigenvalues $\tilde{\lambda}_k$ approximating the same value $\lambda$, it can be advantageous to use a weighted average of these values $\tilde{\lambda}_k$. Assuming that the probability distribution of $\tilde{\mu}_k$ has a constant variance $\sigma_\mu^2$ for all $k$ \footnote{This is only an approximation, since non-dominant eigenvalues (i.e. satisfying $\Re\{\mu_k\} \ll 0$) might be computed by the DMD algorithm with larger errors.},  we can consider the weighted average
\begin{equation}
\label{weighted_av}
\overline{\lambda} = \frac{\sum_{k=1}^m \frac{\tilde{\lambda}_k}{(\Delta(\tilde{\mu}_k))^2}}{\sum_{k=1}^m \frac{1}{(\Delta(\tilde{\mu}_k))^2}}
\end{equation}
which is associated with a probability distribution characterized by the variance
\begin{equation*}
\var{} (\overline{\lambda}) = \frac{\sigma^2}{\sum_{k=1}^m \frac{1}{(\Delta(\tilde{\mu}_k))^2}}\,.
\end{equation*}

\subsection{Nonlinear systems with identical units}
\label{sec:nonlinear}

Now we show that the spectral network identification framework developed in the case of linear systems can easily be extended to nonlinear systems. Assume that the units have a nonlinear dynamics
\begin{equation}
\label{local_dyn2}
\begin{array}{rclc}
\dot{x}_k & = & F(x_k) + G(x_k) u_k & \qquad x_k \in \mathbb{R}^m \\
y_k & = & H(x_k) & \qquad y_k \in \mathbb{R}
\end{array}
\qquad k \in \{1,\dots,n\}
\end{equation}
with the (analytic) functions $F:\mathbb{R}^m \to \mathbb{R}^m$, $G:\mathbb{R}^m \to \mathbb{R}^m$, and $H:\mathbb{R}^m \to \mathbb{R}$. The units interact through the diffusive coupling
\begin{equation}
\label{diff_coupling2}
u_k=\sum_{j=1}^n W_{kj} (y_k-y_j) \,.
\end{equation}
We make the standing assumption that the local dynamics \eqref{local_dyn2} admit a stable fixed point $x^*$ and that the units synchronize, so that the solutions $X(t)$ of \eqref{local_dyn2}-\eqref{diff_coupling2} converge to the (stable) fixed point $X^*=[x^* \dots x^*]^T$. The Jacobian matrix associated with \eqref{local_dyn2}-\eqref{diff_coupling2} linearized at $X^*$ is given by
\begin{equation}
\label{Jacob_K2}
K \triangleq I_n \otimes A - L \otimes B C^T \in \mathbb{R}^{mn \times mn}
\end{equation}
with $A=\partial F/\partial x(x^*)$, $B=G(x^*)$, and $C=\nabla H(x^*)$ ($\nabla$ denotes the gradient).
Since \eqref{Jacob_K2} is similar to \eqref{K}, the following result can be obtained directly from Proposition \ref{prop:net_ident}.
\begin{proposition}
\label{prop:net_ident2}
Consider the local dynamics \eqref{local_dyn2} and assume that $(A,B)=(\partial F/\partial x(x^*),G(x^*))$ and $(A,C)=(\partial F/\partial x(x^*),\nabla H(x^*))$ are controllable and observable pairs, respectively. Then the relationship between the spectrum of the Laplacian matrix $L$ and the spectrum of the Jacobian matrix \eqref{Jacob_K2} is a bijection (i.e. \eqref{rel_K_L} holds). Moreover, the spectrum of $L$ is given by \eqref{sigma_L}.
\end{proposition}
\begin{proof}
The proof follows from Proposition \ref{prop:net_ident}, with $A=J(x^*)$, $B=G(x^*)$, and $C=\nabla H(x^*)$.
\end{proof}

Proposition \ref{prop:net_ident2} implies that the Laplacian eigenvalues of the network can be obtained from the eigenvalues $\mu_j$ of the Jacobian matrix $K$. Moreover, the eigenvalues of $K$ can be obtained from sparse measurement of the network dynamics. In the case of nonlinear systems, they are related to spectral properties of the dynamics defined in the framework of the so-called Koopman operator.

\paragraph{Koopman operator} 
Let $X(t)$ be a trajectory solution of \eqref{local_dyn2}-\eqref{diff_coupling2} associated with the initial condition $X_0$. We suppose that $p \ll n$ measurements of $X(t)$ are obtained through a possibly nonlinear observation function $f:\mathbb{R}^{mn} \to \mathbb{R}^p$, which depends on a few local states in our case.
The Koopman operator(s) $U^t$ are a semi-group acting on the set of such functions $\mathbb{R}^{mn} \to \mathbb{R}^p$. They are defined by 
$$(U^t f)(X_0) = f(X(t))$$
for all $f$ and $X_0$. The value $(U^t f)(X_0)$ is thus the value that will be observed at time $t$ via $f$ for a trajectory which is at $X_0$ at time $t=0$. One can verify that the $U^t$ are always linear and can therefore be characterized by spectral properties.
Provided that $f$ is analytic, the spectral decomposition of the operator yields
\begin{equation}
\label{state_evol_nonlin}
f(X(t))=X^* + \sum_{(j_1,\dots,j_{mn}) \in \mathbb{N}^{mn}} V^f_{j_1,\dots,j_{mn}} \, e^{(j_1 \mu_1+\dots+j_{mn} \mu_{mn})t}\,,
\end{equation}
where $j_1 \mu_1+\dots+j_{mn} \mu_{mn}$ are the eigenvalues of the Koopman operator and $V^f_{j_1,\dots,j_{mn}} \in \mathbb{R}^p$ are the so-called Koopman modes, which depend on $X_0$ \cite{MMM_isostables,Mezic_ann_rev}. In addition, it can be shown that the DMD algorithm extracts the spectral properties of the Koopman operator from snapshots of \eqref{state_evol_nonlin} \cite{Rowley,Tu}. Provided that the dominant Koopman modes are nonzero (see also Remark \ref{rem:nonzero_modes}), the algorithm yields the dominant Koopman eigenvalues, which are the eigenvalues of the Jacobian matrix \eqref{Jacob_K2}; see e.g. \cite{Mezic_ann_rev}. According to Proposition \ref{prop:net_ident2}, these eigenvalues can be used to retrieve the Laplacian eigenvalues.

\begin{remark}
\label{rem:nonzero_modes}
The DMD algorithm can capture dominant eigenvalues $\mu_k$ only if the associated dominant Koopman modes $V^f_{j_1,\dots,j_{mn}}$ (with $j_1+\cdots+j_{mn}=1$) are nonzero in \eqref{state_evol_nonlin}. These dominant modes are given by $V^f_{j_1,\dots,j_{mn}}= \nabla f(X^*)^T V_j$, where $V_j$ are the right eigenvectors of $K$ and with the notation $\nabla f=[\nabla f_1 , \cdots , \nabla f_p]^T$; see e.g. \cite{MMM_isostables}. It follows that the dominant Koopman modes are nonzero if the pair $(K,\nabla f(X^*))$ is observable. In particular, we must have $\nabla f(X^*) \neq 0$. In the following, we will assume that these conditions are satisfied.
\end{remark}

\begin{example}[Nonlinear system] \label{example2} We consider the same network as in Example \ref{example1} (see the adjacency matrix in Appendix \ref{app_examples}) with the local nonlinear dynamics
\begin{equation}
\label{nonlin_dyn_example2}
\begin{array}{ccl}
\mat{c}{\dot{x}_{k,1} \\ \dot{x}_{k,2}} & = & \mat{c}{-x_{k,1}-x^3_{k,1}-2 x_{k,2} \\ x_{k,1}-x_{k,2}-x^3_{k,2}}
+\mat{c}{\cos(x_{k,2}) \\ 1.5 \cos(x_{k,2})} u_k \,, \\
y_k & = & x_{k,1}+x_{k,1}^2+x_{k,2} \,,
\end{array}
\end{equation}
where $x_k=[x_{k,1} \, x_{k,2}]^T$ is the state vector assigned to vertex $k$. We choose the nonlinear observation function $f(X)=[\sin(x_{1,1}+x_{1,2}) \, \sin(x_{1,1}-x_{1,2})]$. Figure \ref{fig:example2}(a) shows that the DMD algorithm applied to times series related to $10$ different initial conditions retrieves the dominant eigenvalues $\mu_k$ of $K$, but cannot compute the eigenvalues with a fast decay rate. However, the eigenvalues are redundant since two distinct eigenvalues are related to the same Laplacian eigenvalue, so that only dominant eigenvalues $\mu_k$ are sufficient to obtain the full Laplacian spectrum. Figure \ref{fig:example2}(b) shows that all the Laplacian eigenvalues are indeed obtained with good accuracy, although two additional values are predicted incorrectly. Note also that when two values are obtained for the same eigenvalue (as it can be observed in Figure \ref{fig:example2}(b)), one could use the weighted average \eqref{weighted_av} for a better approximation.  \hfill $\diamond$
\begin{figure}[htbp]
\centering
\subfigure[Spectrum of $K$]{\includegraphics[width=0.45\columnwidth]{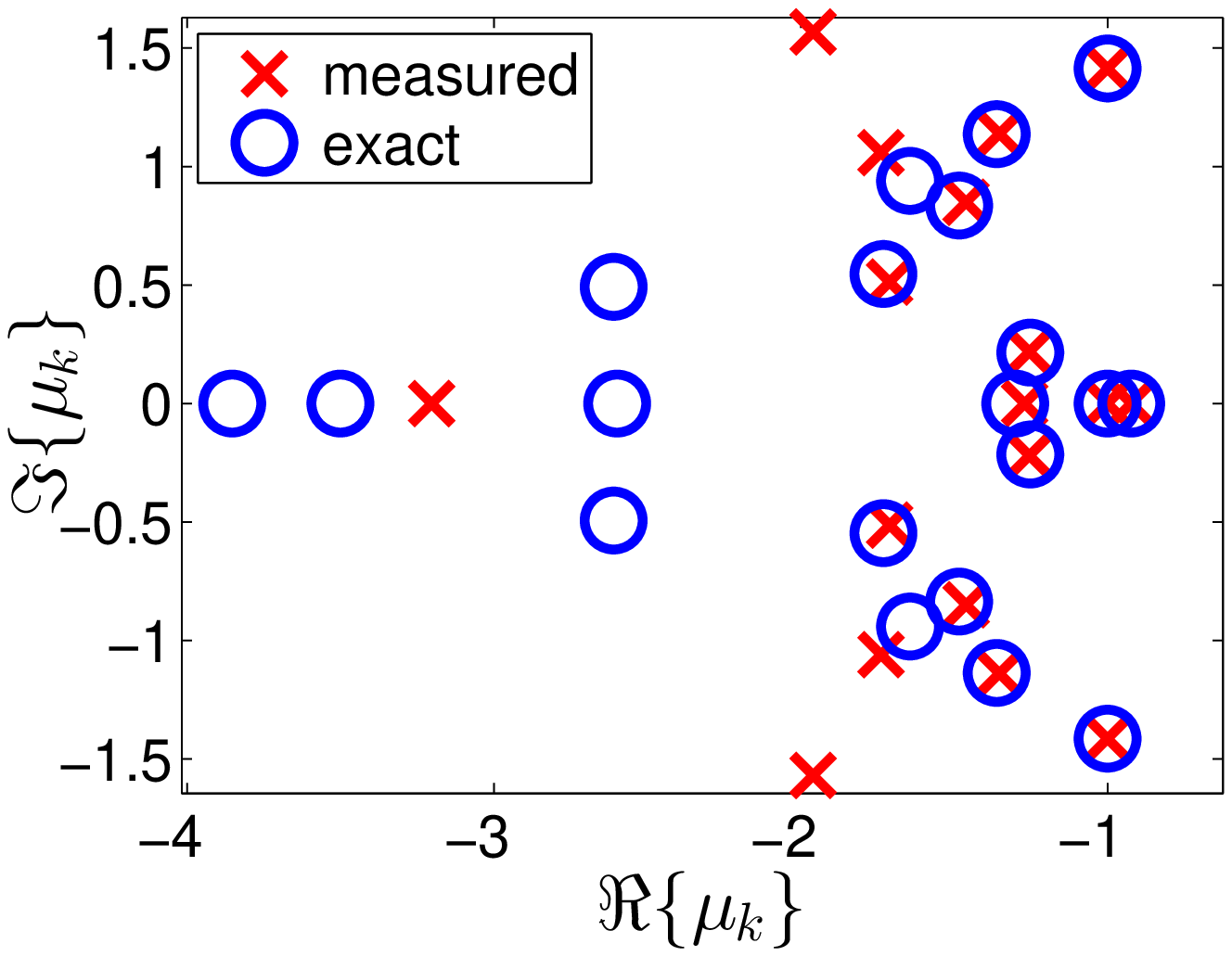}}
\subfigure[Spectrum of $L$]{\includegraphics[width=0.45\columnwidth]{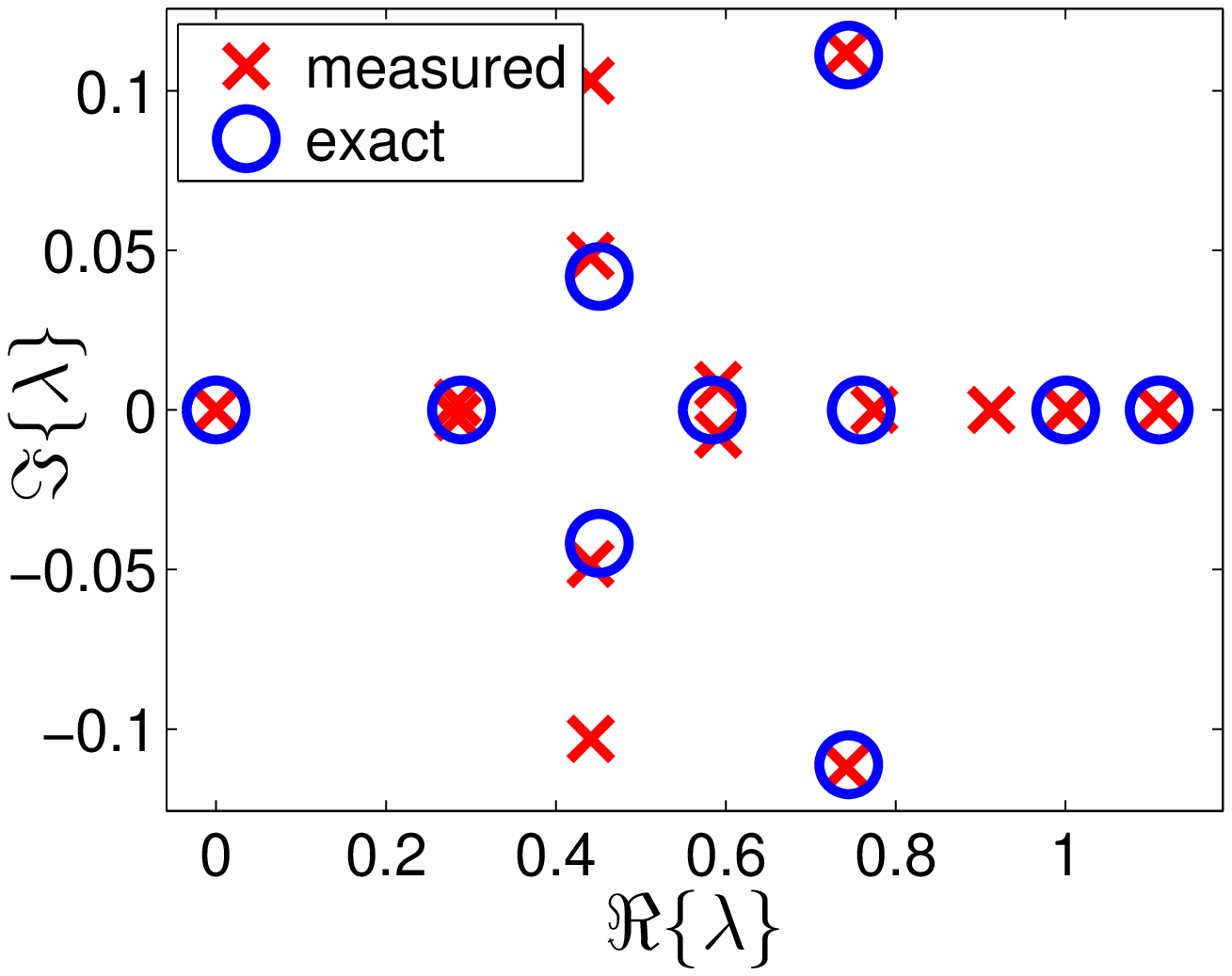}}
\caption{Spectral network identification for a nonlinear system, using the measure of one state at one vertex (Example \ref{example2}). Simulation parameters are given in Appendix \ref{app_examples}. (a) The DMD algorithm yields the dominant eigenvalues $\mu_k$ of the matrix $K$. (b) The Laplacian eigenvalues $\lambda$ of the network are recovered.}
\label{fig:example2}
\end{figure}
\end{example}

\subsection{Non-identical units: impossibility results}
\label{sec:impossibility}

So far we have considered the case of identical units sharing the same local dynamics. We now focus on the case of non-identical units and show that the spectral identification problem cannot be solved in this case. This is illustrated by the following example. Consider the one-dimensional linear local dynamics $\dot{x}_k= A_k x_k + u_k$, $y_k=x_k$ (with $x_k,y_k \in \mathbb{R}$), where $A_k \in \mathbb{R}$ accounts for the heterogeneity of the units dynamics.
The global dynamics of the network is given by
$\dot{X}= (\bar{A}-L) X$, with $\bar{A}=\mathrm{diag}(A_1 \, \cdots \, A_n)$. Let
\begin{equation*}
L_1=\mat{cccc}
     { 1  &  0 & 0 & -1 \\
			 -1 & 3 & -1 & -1 \\
			 -1  & 0 & 2 & -1\\
			 0  & -1 & -1 & 2}
\quad
L_2=\mat{cccc}
{1 & 0 & 0 & -1 \\
 0  & 2 & -1 & -1 \\
 -1 & 0 & 2 & -1 \\
 -1 & -1 & -1 & 3}
\end{equation*}
\begin{equation*}
\bar{A}=-\mat{cccc}
{0.5 & 0 & 0 & 0 \\
0 & 1.5 & 0 & 0 \\
0 & 0 & 1.5 & 0 \\
0 & 0 & 0 & 1.5} \,.
\end{equation*}
Then we have
\begin{equation*}
\sigma(\bar{A}-L_1)=\sigma(\bar{A}-L_2)=\{-0.5877,-2.5,-3.2135,-0.5878 \}\,,
\end{equation*}
but
\begin{equation*}
\sigma(L_1)=\{0,2,3,3\} \neq \sigma(L_2)=\{0,2,2,4\}\,,
\end{equation*}
so that
\begin{equation*}
\sigma(\bar{A}-L_1)= \sigma(\bar{A}-L_2) \nRightarrow \sigma(L_1)= \sigma(L_2)
\end{equation*}
and \eqref{rel_K_L} in Proposition \ref{prop:net_ident} does not hold. It follows that the relation between the set of spectra $\sigma(L)$ and the set of spectra $\sigma(\bar{A}-L)$ is not injective in the case of non-identical units, so that the Laplacian eigenvalues cannot be inferred from the (measured) eigenvalues of $\bar{A}-L$. This example shows that the spectral network identification problem cannot be solved even when the graph is unweighted and when only one unit differs from the others.\\


Now we consider the general nonlinear dynamics of non-identical units
\begin{equation}
\label{local_dyn_nonident}
\begin{array}{rclc}
\dot{x}_k & = & F(x_k) + \delta F_k(x_k) + G(x_k) u_k\\
y_k & = & H(x_k)
\end{array}
\qquad k \in \{1,\dots,n\}\,,
\end{equation}
where $\delta F_k$ accounts for the heterogeneity of the units dynamics. We assume that the units are almost identical, so that $\|\delta F_k\| \ll 1$. Note that the units do not synchronize perfectly, since the system admits a global fixed point $X^*+\delta X^*=[x^*+\delta x_1^* \, \cdots \, x^*+\delta x_n^*]$ (with $\|\delta x_k^*\| \ll 1$). We consider that the functions $G$ and $H$ are identical for all the units. There is almost no loss of generality, since these functions describe the connections between units, and these connections are already heterogeneous in the case of weighted graphs. In addition, this simplification does not significantly affect the theoretical results developed in the remaining of the paper.

The Jacobian matrix related to the system \eqref{local_dyn_nonident} (linearized around $X^*+\delta X^*$) is
\begin{equation}
\label{delta_K}
K+\mathrm{diag}(\delta A_1 \cdots \delta A_n) \triangleq K + \delta K
\end{equation}
where $K$ is given by \eqref{Jacob_K2} and where $\delta A_k$ is such that
\begin{equation*}
\frac{\partial F}{\partial x}(x^*+\delta x_k^*)+\frac{\partial \delta F_k}{\partial x}(x^*+\delta x_k^*) = A + \delta A_k\,,
\end{equation*}
with $A=\partial F/\partial x(x^*)$. Since the DMD algorithm provides the eigenvalues of $K+\delta K$, the spectral identification problem is to compute $\sigma(L)$ from $\sigma(K+\delta K)$. Equivalently, since the relationship between $\sigma(K)$ and $\sigma(L)$ has been established in Section \ref{sec:lin_ident} and \ref{sec:nonlinear}, one has to estimate the eigenvalues of $K$ from the perturbed eigenvalues of $K+\delta K$ (note that $\|\delta K\| \ll 1$ since $|\delta F_k|\ll 1$). In the case of unweighted graphs, it can be shown that $\sigma(K_1+\delta K)=\sigma(K_2+\delta K) \Rightarrow \sigma(K_1)=\sigma(K_2)$ if the perturbation $\delta K$ is small enough, so that the spectral identification problem can be solved exactly. However this situation is very restrictive. For more general graphs, perturbation theory for matrix eigenvalues \cite{Perturbation_eigenvalues_book} can provide upper bounds on the difference between the eigenvalues of $K$ and $K+\delta K$, but these bounds are too conservative, especially when the network is large.

\alex{It is also noticeable that the linearized dynamics of identical agents that do not synchronize but converge to different equilibria are in general equivalent to the (linear) dynamics of non-identical agents. In such a case, we would thus encounter similar issues when inferring the spectral properties of the network .}

In the next section, we show that a statistical approach can circumvent these limitations in the case of large networks.


\section{Statistical approach to large networks}
\label{sec:stat}

In the case of large graphs, most individual Laplacian eigenvalues have an insignificant influence on the network dynamics, making them very hard to identify precisely. On the other hand, each of them taken individually only captures a very small amount of information about the network structure. Therefore, recovering each individual eigenvalue is on the one hand impractical but on the other hand not really necessary.
Instead, it is much more relevant and convenient to focus on statistical measures of the spectral density of the Laplacian matrix. In this section, we show that one can estimate the first spectral moments \eqref{spec_moments} of the Laplacian matrix from sparse measurements in the network. These spectral moments are related to statistical information on the degree distribution of the vertices (see Section \ref{sec:exact}). This approach is also well-suited to the case of non-identical units and can be used to obtain some information on the unweighted underlying graph.

\subsection{Spectral moments of the Laplacian matrix}
\label{sec:stat_identical}

From a few eigenvalues of $\sigma(K)$ obtained with the DMD algorithm, one can compute the spectral moments of $K$; see Section \ref{sec:approx_spec_mom} for details on the numerical method. Then the spectral moments of $L$ can be obtained from the spectral moments of $K$, as shown in the following proposition.
\begin{proposition}
\label{prop:moments}
Suppose that $K=I_n \otimes A - L \otimes BC^T$. Then the spectral moments of $L$ are given by
\begin{equation*}
\mathcal{M}_k(L)=\frac{(-1)^k}{(C^TB)^k} \left(m \mathcal{M}_k(K)+\sum_{j=0}^{k-1} {k \choose j} (-1)^{j+1} \mathcal{M}_j(L) \, \tr(A^{k-j}(BC^T)^j)\right)\,,
\end{equation*}
with ${k \choose j} = \frac{k!}{j!(k-j)!}$.
\end{proposition}
\begin{proof}
We have
\begin{equation*}
\begin{split}
\mathcal{M}_k(K) & = \frac{1}{mn} \tr (I_n \otimes A - L \otimes BC^T)^k \\
& = \frac{1}{mn} \sum_{j=0}^k {k \choose j} \tr{} \left((I_n \otimes A )^{k-j} (-L \otimes BC^T)^j \right) \\
& = \frac{1}{mn} \sum_{j=0}^k {k \choose j} \tr{} \left((-L)^j \otimes A^{k-j} (BC^T)^j \right) \\
& = \frac{1}{m} \sum_{j=0}^k {k \choose j} (-1)^j \mathcal{M}_j(L) \, \tr{} \left(A^{k-j} (BC^T)^j\right) \\
& = \frac{(-1)^k}{m} \mathcal{M}_k(L) (C^TB)^k + \frac{1}{m} \sum_{j=0}^{k-1} {k \choose j} (-1)^j \mathcal{M}_j(L) \tr{} \left(A^{k-j} (BC^T)^j\right)
\end{split}
\end{equation*}
and the result follows.
\end{proof}

Note that this result holds for both cases of linear and nonlinear units. In the nonlinear case, we have $A=\partial F/\partial x(x^*)$, $B=G(x^*)$, and $C=\nabla H(x^*)$.

We will focus on the first two moments and it follows from Proposition \ref{prop:moments} that
\begin{equation}
\label{M1_L}
\mathcal{M}_1(L)=\frac{-m \mathcal{M}_1(K)+\tr (A)}{C^TB}
\end{equation}
and
\begin{equation}
\label{M2_L}
\mathcal{M}_2(L) = \frac{m \mathcal{M}_2(K)- \tr (A^2)+2 \mathcal{M}_1(L) \, \tr (ABC^T)}{(C^TB)^2} \,.
\end{equation}

\begin{example}[Large network]
\label{example3}
We consider a random Erd{\H o}s-R\'enyi graph with $n=100$ vertices and a probability $p_{edge}=0.3$ (supposedly not known) for any two vertices to be connected. The weights of the edges are randomly distributed according to a uniform distribution on $[0,0.1]$. The local dynamics of the units is given by \eqref{lin_dyn_example1} and the observation function is $f(X)=x_{1,1}$, corresponding to the measurement of one state of one unit. \alex{Note that the dynamics are linear, a choice made to provide the best illustration of the theoretical results. Examples of identification of large networks with nonlinear dynamics are shown in Section \ref{sec:examples}.}

We use a heuristic method to estimate the spectral moments of $K$ (Figure \ref{fig:example3}), approximating the clusters of eigenvalues by the convex hull of values obtained with the DMD algorithm and assuming a uniform distribution of the eigenvalues within each cluster; see Section \ref{sec:approx_spec_mom} for more details on the method. Using \eqref{M1_L} and \eqref{M2_L}, we finally obtain the spectral moments $\mathcal{M}_1(L)=1.50$ and $\mathcal{M}_2(L)=2.35$, which are close to the exact values. From \eqref{mean_deg} and \eqref{quad_mean_deg}, one obtains good approximations of the mean degree $\mathcal{D}_1(\mathcal{G})$ and quadratic mean degree $\mathcal{D}_2(\mathcal{G})$. The results are summarized in Table \ref{tab:example3}. We also performed similar simulations for $100$ different random networks and computed in an automatic way the spectral moments of the Laplacian matrix. Table \ref{tab:stat_example3} shows that the mean relative error is of the order of $10\%$.\hfill $\diamond$

\begin{figure}[htbp]
\centering
\includegraphics[width=8cm]{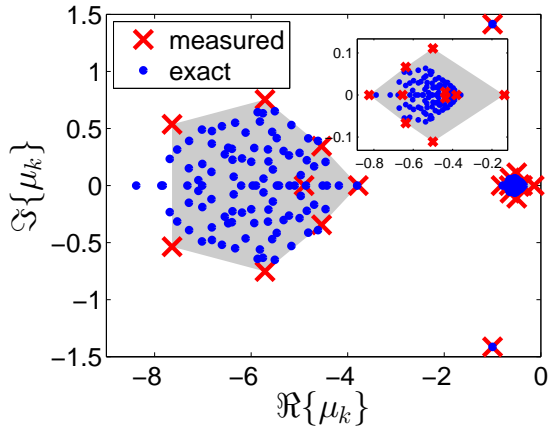}
\caption{Identification of the eigenvalues of $K$ in the case of a large network using measurements at one vertex. Simulation parameters are given in Appendix \ref{app_examples}. The convex hull (in gray) of the measured eigenvalues provides a good approximation of the spectral moments of $K$. The inset shows the cluster of eigenvalues lying near the origin.}
\label{fig:example3}
\end{figure}

\end{example}

\begin{table}[h]
	\centering
		\begin{tabular}{ccc}
		\hline
		& estimated & exact \\
		\hline
		$\mathcal{M}_1(K)$ & -3.25 & -3.23 \\
		$\mathcal{M}_1(L)$ & 1.50 & 1.48 \\
		$\mathcal{D}_1(\mathcal{G})$ & 1.50 & 1.48 \\
		\hline
		\end{tabular}
		\hspace{1cm}
		\begin{tabular}{ccc}
		\hline
		& estimated & exact \\
		\hline
		$\mathcal{M}_2(K)$ & 18.59 & 18.23 \\
		$\mathcal{M}_2(L)$ & 2.35 & 2.30 \\
		$\mathcal{D}_2(\mathcal{G})$ & $2.26 \leq \cdots \leq 2.35$ & 2.27\\
		\hline
		\end{tabular}
		\caption{Identification of the spectral moments and mean vertex degrees of a large graph.}
		\label{tab:example3}
\end{table}

\begin{table}[h]
	\centering
		\begin{tabular}{lcc}
		\hline
		  & $\mathcal{M}_1(L)$ & $\mathcal{M}_2(L)$ \\
			\hline
			mean error (absolute) & 0.10 & 0.29 \\
			mean error (relative) & 0.07 & 0.13 \\
		root mean squared error (absolute) & 0.13 & 0.37 \\
		root mean squared error (relative) & 0.09 & 0.16 \\
		\hline
		\end{tabular}
		\caption{The spectral moments $\mathcal{M}_1(L)$ and $\mathcal{M}_2(L)$ are estimated in an automatic way for $100$ different random networks. The mean relative error is of the order of $10\%$.}
		\label{tab:stat_example3}
\end{table}

\subsection{Estimation of the spectral moments with non-identical units}

Now we assume that the units are not identical and that their local dynamics are randomly distributed with known mean and variance (we will comment on the case of unknown mean and variance in Remark \ref{rem:dyn_measured}).  
We can then obtain an estimation of the spectral properties of the dynamics produced by the same network but with identical units, so that the results of Section \ref{sec:stat_identical} can be used to compute the moments of the Laplacian matrix.

We consider the local dynamics \eqref{local_dyn_nonident}, with the objective of estimating the spectral moments of $L$ from the spectral moments of $K+\delta K$ (see \eqref{delta_K}). The matrix $\delta K$ is a random block-diagonal matrix and the nonzero entries of $\delta K$ are assumed to be independent random variables of zero mean and standard deviation $s$, i.e. $\delta K = \mathrm{diag}(\delta A_1 \cdots \delta A_n)$ with 
\begin{equation}
\label{random_delta_K}
 \qquad \begin{array}{rcll}
\mathbb{E}([\delta A_k]_{ij}) & = & 0  \\
\mathbb{E}([\delta A_k]_{ij}^2) & = & s^2 \\
\mathbb{E}[[\delta A_k]_{ij} [\delta A_k]_{i'j'}] & = & 0 & \forall (i,j)\neq (i',j')
\end{array}
\end{equation}
for $k=1,\dots,n$ and $i,j =1,\dots,m$.

The spectral moments of $K$ are related to the moments of the averaged spectral density of $K+\delta K$. In particular, we will show that $\mathcal{M}_1(K)=\mathbb{E}(\mathcal{M}_1(K+\delta K))$ and $\mathcal{M}_2(K)=\mathbb{E}(\mathcal{M}_2(K+\delta K))-s^2$. Then estimations $\widehat{\mathcal{M}_1}(L)$ and $\widehat{\mathcal{M}_2}(L)$ of the spectral moments of $L$ can be computed by considering the measured (random) values $\mathcal{M}_1(K+\delta K)$ and $\mathcal{M}_2(K+\delta K)$ instead of $\mathcal{M}_1(K)$ and $\mathcal{M}_2(K)$. The expectation of $\widehat{\mathcal{M}_1}(L)$ and $\widehat{\mathcal{M}_2}(L)$ (with respect to the randomness on $\delta K$) is equal to the exact spectral moments of $L$. These results are summarized in the following proposition. The proof is given in Appendix \ref{sec:app_proofs}.
\begin{proposition}
\label{prop_rand_delta_K}
Assume that $\delta K$ is a block-diagonal matrix that satisfies \eqref{random_delta_K} and consider
\begin{eqnarray}
\widehat{\mathcal{M}_1}(L) & = & \frac{-m \mathcal{M}_1(K+\delta K) +\tr (A)}{C^TB} \label{M1_estim} \\
\widehat{\mathcal{M}_2}(L) & =  & \frac{m \mathcal{M}_2(K+\delta K) -ms^2 -\tr (A^2) + 2 \widehat{\mathcal{M}_1}(L)\tr (ABC^T)}{(C^TB)^2}\,. \label{M2_estim}
\end{eqnarray}
Then
\begin{eqnarray*}
\mathbb{E}(\widehat{\mathcal{M}_1}(L)) & = & \mathcal{M}_1(L) \\
\mathbb{E}(\widehat{\mathcal{M}_2}(L)) & = & \mathcal{M}_2(L) \,.
\end{eqnarray*}
Moreover, we have
\begin{eqnarray*}
\var(\widehat{\mathcal{M}_1}(L))  & =  & \frac{m s^2}{n (C^TB)^2} \\
\var(\widehat{\mathcal{M}_2}(L)) & = & \mathcal{O} \left(\frac{1}{n}+ \frac{\mathcal{D}_1(\mathcal{G})}{n} + \frac{\mathcal{D}_2(\mathcal{G})}{n} \right)
\end{eqnarray*}
where $\mathcal{D}_1(\mathcal{G})$ and $\mathcal{D}_2(\mathcal{G})$ are the mean vertex degree and the quadratic mean vertex degree, respectively.
\end{proposition}

In Proposition \ref{prop_rand_delta_K}, the estimated values $\widehat{\mathcal{M}_1}(L)$ and $\widehat{\mathcal{M}_2}(L)$ are assumed to be obtained with a perfect estimation of the moments of $K+\delta K$ (computed from the eigenvalues obtained with the DMD algorithm), but this is not the case in practice. For large networks, the variance of the estimated moments \eqref{M1_estim} and \eqref{M2_estim} is significantly smaller than the error on the estimation of the moments of $K+\delta K$, so that \eqref{M1_estim}-\eqref{M2_estim} provide approximations of the spectral moments of $L$ that are almost as good as in the case of identical units.

\begin{remark}[Unknown local dynamics]
\label{rem:dyn_measured}
If we suppose that the average local dynamics of the units is not known (i.e. $F$ is not known in \eqref{local_dyn_nonident}), then $A=\partial F/\partial x (x^*)$ is estimated by $A+\delta A$ (e.g. measured on one unit), where $\mathbb{E}(\delta A_{ij})=0$ and $\mathbb{E}(\delta A_{ij}^2)=s^2$ for all $i,j$. In particular, $\mathbb{E}(\tr (A+\delta A))=\tr (A)$ and $\mathbb{E}(\tr (A+\delta A)^2)=\tr (A^2) + m s^2$ and we have $\mathbb{E}(\widehat{\mathcal{M'}_1}(L))=\mathcal{M}_1(L)$ and $\mathbb{E}(\widehat{\mathcal{M'}_2}(L))=\mathcal{M}_2(L)$, with
\begin{eqnarray*}
\widehat{\mathcal{M'}_1}(L) & = & \frac{-m \mathcal{M}_1(K+\delta K) +\tr (A+\delta A)}{C^TB} \\
\widehat{\mathcal{M'}_2}(L) & =  & \frac{m \mathcal{M}_2(K+\delta K) -\tr (A+\delta A)^2 + 2 \widehat{\mathcal{M}_1'}(L)\tr (ABC^T)}{(C^TB)^2}\,.
\end{eqnarray*}
We remark that the variance $s^2$ of the distribution of dynamics does not need to be known in this case.
\end{remark}

\begin{remark}[Other distributions and two different populations] So far we have considered that all nonzero entries of $\delta K$ are independently and identically distributed. We can also have another situation where the values of the subsystems are strongly correlated. Consider the random matrix
\begin{equation*}
\delta K = \textrm{diag}(\varepsilon_1 \delta A,\dots, \varepsilon_n \delta A) = \mathcal{E} \otimes \delta A
\end{equation*}
where $\mathcal{E}$ is a random diagonal matrix whose elements $\varepsilon_i$ are randomly distributed and satisfy $\mathbb{E}(\varepsilon_i)=0$, $\mathbb{E}(\varepsilon_i^2)=s^2$, and $\mathbb{E}(\varepsilon_i \varepsilon_j)=0$ for $i\neq j$. In this case, we can show that $\mathbb{E}(\widehat{\mathcal{M''}_1}(L))=\mathcal{M}_1(L)$ and $\mathbb{E}(\widehat{\mathcal{M''}_2}(L))=\mathcal{M}_2(L)$, with
\begin{equation}
\label{moments_2_pop}
\begin{array}{rcl}
\widehat{\mathcal{M''}_1}(L) & = & \displaystyle \frac{-m \mathcal{M}_1(K+\delta K) +\tr (A)}{C^TB} \\
\widehat{\mathcal{M''}_2}(L) & =  & \displaystyle \frac{m \mathcal{M}_2(K+\delta K) -s^2 \frac{\tr(\delta A^2)}{m} -\tr (A^2) + 2 \widehat{\mathcal{M}_1''}(L)\tr (ABC^T)}{(C^TB)^2}
\end{array}
\end{equation}
and the variances are similar to the variances obtained in Proposition \ref{prop_rand_delta_K}.

When the probability distribution function of $\varepsilon_i$ is a (weighted) sum of two Dirac functions at $\varepsilon=\pm 1$, the local dynamics of the units is randomly chosen between two possible dynamics, i.e. $\dot{x}_k=F(x_k) + \delta F(x_k)+G(x_k) u_k$ and $\dot{x}_k=F(x_k) - \delta F(x_k)+G(x_k) u_k$. In this case, the spectral moments are given by \eqref{moments_2_pop} with $s=1$. \hfill $\diamond$
\end{remark}

We could also consider a more general situation, where the subsystems $\delta A_k$ are independently and identically distributed. This is however out of the scope of this paper.

\begin{example}[Non-identical units]
\label{example4}
We consider the same network as in Example \ref{example3} but assume that the dynamics of the units are randomly distributed with a standard deviation $s=0.2$, i.e. we consider the system $\dot{X}=(K+\delta K)X$ with the entries of $\delta K$ satisfying \eqref{random_delta_K}. With the measurement of one state of one unit, the DMD algorithm provides an approximation of the spectrum of the random matrix $K+\delta K$ (Figure \ref{fig:example4}). Using \eqref{M1_estim} and \eqref{M2_estim}, we obtain an accurate estimation $\widehat{\mathcal{M}_1}(L)=1.52$ and $\widehat{\mathcal{M}_2}(L)=2.15$. The results are summarized in Table \ref{tab:example4}. Note that the variances satisfy $\var(\widehat{\mathcal{M}_1}(L))<0.0001$ and $\var(\widehat{\mathcal{M}_2}(L))<0.1$, and are negligible with respect to the error induced by the computation of the moments $\mathcal{M}_1(K+\delta K)$ and $\mathcal{M}_2(K+\delta K)$. \hfill $\diamond$

\begin{figure}[htbp]
\centering
\includegraphics[width=8cm]{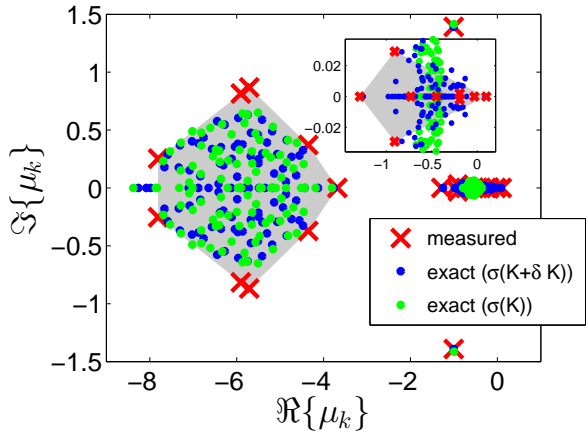}
\caption{Identification of the spectrum of $K+\delta K$ in the case of non-identical units. Simulation parameters are given in Appendix \ref{app_examples}. The eigenvalues of $K+\delta K$ (blue dots) are different from the eigenvalues of $K$ (green dots). They are approximated by eigenvalues obtained with the DMD algorithm (red crosses), whose convex hull (in gray) is used for the computation of the spectral moments of $K+\delta K$. The inset shows the cluster of eigenvalues lying near the origin.}
\label{fig:example4}
\end{figure}

\begin{table}[h]
	\centering
		\begin{tabular}{ccc}
		\hline
		& estimated & exact \\
		\hline
		$\mathcal{M}_1(K+\delta K)$ & -3.28 & -3.22 \\
		$\mathcal{M}_1(K)$ & -3.28 & -3.23 \\
		$\mathcal{M}_1(L)$ & 1.52 & 1.48 \\
		\hline
		\end{tabular}
		\hspace{1cm}
	  \begin{tabular}{ccc}
		\hline
		& measured & exact \\
		\hline
		$\mathcal{M}_2(K+\delta K)$ & 17.80 & 18.17 \\
		$\mathcal{M}_2(K)$ & 17.76 & 18.23 \\
		$\mathcal{M}_2(L)$ & 2.15 & 2.30 \\
		\hline
		\end{tabular}
		\caption{Identification of the spectral moments of a large graph with a heterogeneous population.}
		\label{tab:example4}
\end{table}

\end{example}

\subsection{Spectral moments of the unweighted Laplacian matrix}

We have considered so far the general case of weighted graphs $\mathcal{G}$. It is also of interest to infer the spectral properties of unweighted graphs $\bar{\mathcal{G}}$, which have the same vertex and edge sets as $\mathcal{G}$ but a weight function $\bar{w}(i,j)=1$ for all $(i,j)\in E$. The degree $\bar{d}_k$ of a vertex $k$ of $\bar{\mathcal{G}}$ corresponds to the number of connections of the unit $k$, and the spectral moments $\mathcal{M}(\bar{L})$ of the unweighted Laplacian matrix $\bar{L}$ are related to the distribution of this number of connections.


We consider that the weights $W_{ij}$ are independent random variables of mean $r>0$ and standard deviation $s>0$, i.e.
\begin{equation}
\label{random_weights}
\begin{array}{rcl}
\mathbb{E}(W_{ij}) & = & r  \\
\mathbb{E}(W_{ij}^2) & = & r^2+ s^2  \\
\mathbb{E}[W_{ij} W_{i'j'}] & = & r^2 \qquad \forall (i,j)\neq (i',j')
\end{array}
\end{equation}
for all $(i,j) \in E$. In this case, the Laplacian matrix $L$ can be seen as a random matrix. Provided that $r$ and $s$ are known, one can estimate the spectral moments of the unweighted Laplacian matrix $\bar{L}$ from the expectation of the spectral moments of the weighted Laplacian matrix $L$. The proof can be found in Appendix \ref{sec:app_proofs}.
\begin{proposition}
\label{prop:rand_weights}
If the edges of the network have random weights distributed according to \eqref{random_weights}, then
\begin{eqnarray}
\displaystyle
\mathcal{M}_1(\bar{L}) & = & \mathbb{E}\left(\frac{\mathcal{M}_1(L)}{r}\right) \,, \label{M1_unw}\\
\displaystyle
\mathcal{M}_2(\bar{L})& = & \mathbb{E}\left(\frac{\mathcal{M}_2(L)}{r^2}-\frac{s^2 \, \mathcal{M}_1(L)}{r^3}\right) \label{M2_unw}\,.
\end{eqnarray}
Moreover, we have
\begin{equation*}
\var{\left(\frac{\mathcal{M}_1(L)}{r}\right)} =\frac{s^2}{r^2} \frac{\mathcal{D}_1(\bar{\mathcal{G}})}{n} < \frac{s^2}{r^2}\,.
\end{equation*}
\end{proposition}

Spectral network identification can be used to obtain the spectral moments of the unweighted Laplacian matrix, which in turn provide information on the distribution of the number of connections in the network (see \eqref{mean_deg} and \eqref{quad_mean_deg}). When $s/r\ll1$, the variance on $\mathcal{M}_1(L)/r$ is small and the present identification framework can provide an accurate estimation of the mean number of connections. This is illustrated in the following example.

\begin{example}
\label{example5}
Consider the same network and dynamics as in Example \ref{example3}. Since the weights are uniformly distributed on $[0,0.1]$, we have $r=0.05$ and $s^2=r^2/3$ in \eqref{random_weights}. Using \eqref{M1_unw} and \eqref{M2_unw} with the values estimated in Example \ref{example3}, we obtain $\mathcal{M}_1(\bar{L})=30.03$ (exact value: $\mathcal{M}_1(\bar{L})=29.77$) and $\mathcal{M}_2(\bar{L})=930.55$ (exact value: $\mathcal{M}_2(\bar{L})=916.23$). It follows from \eqref{mean_deg} and \eqref{quad_mean_deg} that the mean number of connections is estimated as $\mathcal{D}_1(\bar{\mathcal{G}})=30.03$ (exact value: $\mathcal{D}_1(\bar{\mathcal{G}})=29.77$) and the quadratic mean number of connections is bounded by $902.07 \leq \mathcal{D}_2(\bar{\mathcal{G}}) \leq 930.55$ (exact value: $\mathcal{D}_2(\bar{\mathcal{G}})=907.33$).

We can now compare the result obtained with spectral network identification to the result obtained with a direct observation of a random vertex. Performing spectral network identification for $100$ different networks (see results in Table \ref{tab:stat_example3}), we compute the mean squared error on $\mathcal{D}_1(\bar{\mathcal{G}})$, which is equal to $7$. (Indeed, the variance $\var{(\mathcal{M}_1(L)/r)}=0.1$ is small and, according to Table \ref{tab:stat_example3}, the mean squared error on $\mathcal{M}_1(L)/r$ is $(0.13/0.05)^2=6.76$.)
In contrast, the mean squared error obtained with a direct measurement on a random vertex is equal to the variance of the degree distribution, which is given by $n p_{edge}(1-p_{edge})=21$ since the distribution is binomial. Our spectral network identification approach yields thus here a mean squared error three times smaller than a direct observation.\hfill $\diamond$
\end{example}

As suggested in the example above, using the present identification framework might be more accurate for computing the mean number of connections than a direct observation of a small group of randomly chosen vertices, in particular when the degree distribution is characterized by a large variance. For instance, we have considered random networks of $n = 100$ vertices whose number of connections (i.e. unweighted degree) is randomly chosen according to a Gaussian distribution with a large standard deviation (mean: $34$, standard deviation: $28$). The weights of the edges are randomly distributed according to a uniform distribution on $[0, 0.01]$. For the same local dynamics as in Example \ref{example5}, we use spectral network identification to estimate the mean vertex degree of $100$ networks, from measurements at one vertex. We obtain a mean squared error equal to $73$. When the number of connections is observed directly on a random vertex, the mean squared error is $28^2=784$. It would be necessary in this case to measure and average the number of connections of more than $10$ vertices in order to obtain the same precision as the spectral network identification framework (requiring the measure of only one vertex). \alex{This difference could be even more significant with other complex random graphs characterized by a higher variance of the degree distribution (e.g. sparse power-law random graphs).}

\section{Numerical implementation}
\label{sec:numerical}

This section describes the numerical implementation of the spectral network identification framework and proposes a few guidelines mainly based on empirical observations on the performances of the DMD algorithm. The complete numerical procedure for spectral network identification is summarized in Appendix \ref{app_algo}.

\subsection{Time series and data matrix}

The network dynamics is measured through the observation function $f:\mathbb{R}^{mn} \to \mathbb{R}^p$, for $r$ different initial conditions. Then the measurements yield the data points
\begin{equation}
\label{time_series}
z_k = \mat{c}{f(X^{(1)}(k \Delta t)) \\
\vdots \\
f(X^{(r)}(k \Delta t))} \in \mathbb{R}^{rp}\,, \qquad k=0,\dots,K
\end{equation}
where $X^{(l)}(t)$, $l=1,\dots,r$, is a trajectory of the system and $\Delta t$ is the sampling period. From the data points, we construct the data matrix
\begin{equation}
\label{Z}
Z=\mat{cccc} 
{| & | & & | \\
z_0 & z_1 & \cdots & z_K \\
| & | & & |}
\end{equation}
to be used as an input to the DMD algorithm. The number of columns of the matrix is equal to the number $K+1$ of snapshots and the number of rows is equal to the number $r$ of initial conditions multiplied by the number $p$ of observations. \alex{The number of eigenvalues obtained with the DMD algorithm is equal to the number of rows.
If this number is small, it should be increased to improve the efficiency of the DMD algorithm and compute more eigenvalues. This can be done without additional experiments or observations, by considering delays. We can indeed use the fact that,} if $X(t)$ is a trajectory of the system (associated with an initial condition $X(0)$), then $X(t+\delta \,\Delta t)$ is another trajectory of the system (associated with the initial condition $X(\delta \Delta t)$). One can therefore construct a new data matrix
with $crp$ rows (and a few less columns) by considering $c$ sequences of data points shifted in $\delta$ increments, i.e. 
$\{z_0, \dots, z_{K-(c-1)\delta}\}$, $\{z_{\delta}, \dots,z_{K-(c-2)\delta}\}$, \dots , $\{z_{(c-1)\delta},\dots, z_K\}$.
\alex{It is clear that using $n$ delayed versions of the same time series does not bring more information, but in theory, it allows to retrieve $n$ eigenvalues with only one observation, a fact which is related in spirit to Whitney embedding theorem. In practice, however, better numerical results are obtained with several experiments and less delayed time series. It also appeared in our simulations} that $c=2$ or $c=3$ shifted sequences yield the best results. However, the effect of $c$ on the performance of the algorithm has not been extensively studied, and might depend on the other simulation parameters.

\subsection{DMD algorithm}

A detailed description of the DMD algorithm is given in Appendix \ref{app_algo}. The input is the data matrix $Z=[z_0 \, z_1 \, \cdots \, z_{q_2}]$ constructed from a sequence of data points $z_j\in \mathbb{R}^{q_1}$ (note that $q_1=crp$ and $q_2=K-(c-1)\delta$ according to the previous section). The DMD algorithm computes a matrix $T \in \mathbb{R}^{q_1 \times q_1}$ such that
\begin{equation*}
z_{j+1} \approx T z_j \quad \forall j \in \{0,\cdots,q_2-1\} \,.
\end{equation*}
In particular, the algorithm yields the so-called DMD modes $\phi_k$ and eigenvalues $\tilde{\nu}_k$, which are the eigenvectors and eigenvalues of $T$, and it follows that
\begin{equation}
\label{DMD_expan}
z_j \approx \sum_{k=1}^{q_1} \phi_k \tilde{\nu}_k^j\,.
\end{equation}
Comparing \eqref{DMD_expan} with \eqref{state_evol_nonlin}, we can see that the Koopman modes are equal to the DMD modes $\phi_k$, and the Koopman eigenvalues are given by $\tilde{\mu}_k=\log(\tilde{\nu}_k)/\Delta t$.

\subsection{Approximating spectral moments with few eigenvalues}
\label{sec:approx_spec_mom}

In the case of large networks, the eigenvalues $\tilde{\mu}_k$ computed with the DMD algorithm are not equal to the exact eigenvalues $\mu_k$. In particular, the number $q_1=crp$ of eigenvalues $\tilde{\mu}_k$ is much smaller than the number $mn$ of eigenvalues $\mu_k$. However, it is observed that the convex hulls of clusters of eigenvalues $\tilde{\mu}_k$ are good approximations of clusters of exact eigenvalues $\mu_k$. There are typically $n_c=m$ clusters (or sometimes less if they overlap) that can be identified using k-means clustering. Denoting by $\mathcal{S}_j$, $j=1,\dots,n_c$, the convex hulls of clusters of measured eigenvalues $\tilde{\mu}_k$, we can compute the moments of area
\begin{equation}
\label{moments_area}
\begin{array}{rcll}
\displaystyle
\mathcal{A}(\mathcal{S}_j) & = & \displaystyle \int_{\mathcal{S}_j} \, dx dy &\qquad \textrm{area} \\
I_x(\mathcal{S}_j) & = & \displaystyle \frac{1}{\mathcal{A}(\mathcal{S}_j)} \int_{\mathcal{S}_j} x \, dx dy & \qquad \textrm{centroid (with respect to $x$-axis)}\\
I_{xx}(\mathcal{S}_j) & = & \displaystyle \int_{\mathcal{S}_j} y^2 \, dx dy & \qquad \textrm{second moment of area (with respect to $x$-axis)} \\
I_{yy}(\mathcal{S}_j) & = & \displaystyle \int_{\mathcal{S}_j} x^2 \, dx dy &  \qquad \textrm{second moment of area (with respect to $y$-axis)}
\end{array}
\end{equation}
where an eigenvalue $\mu \in \mathcal{S}_j$ is assigned the coordinates $(x,y)=(\Re\{\mu\},\Im\{\mu\})$. Assuming that the eigenvalues are uniformly distributed in the clusters, we obtain the approximation of the spectral moments
\begin{equation}
\label{area_moment_K}
\begin{array}{rcl}
\mathcal{M}_1(K) & \approx & \displaystyle \frac{1}{n_c} \sum_{j=1}^{n_c} I_x(\mathcal{S}_j) \\
\mathcal{M}_2(K) & \approx & \displaystyle \frac{1}{n_c} \sum_{j=1}^{n_c} \frac{I_{yy}(\mathcal{S}_j)-I_{xx}(\mathcal{S}_j)}{\mathcal{A}(\mathcal{S}_j)}
\end{array}
\end{equation}
where we used the fact that there is the same number of eigenvalues in each cluster and that $\mu \in \bigcup_j \mathcal{S}_j \Rightarrow \mu^c \in \bigcup_j \mathcal{S}_j$ (so that $\sum_k \mu_k^2 = \sum_k \Re\{\mu_k\}^2-\sum_k \Im\{\mu_k\}^2$).

\section{Applications}
\label{sec:examples}

In this section, spectral network identification is illustrated with several applications. It is shown that, with a few local measurements in the (possibly large) network, one can estimate the minimum, maximum and average number connections, detect the addition of a new vertex, and measure whether two units influence each other indirectly, and belong thus to the same connected component of the network.

\subsection{Minimum and maximum number of connections}
\label{sec:example6}

The spectral network identification framework is efficient to approximate the minimum and maximum number of connections in an undirected graph. It can also be used to detect the addition of a vertex that is weakly connected to the rest of the network.

We aim at identifying a network that was constructed as a random Erd{\H o}s-R\'enyi graph of $n=100$ vertices for which there is an edge between any pair of vertices with a probability $p_{edge}=0.3$ (supposedly not known). The graph is undirected and unweighted. The local dynamics of the units is given by
\begin{equation*}
\begin{array}{ccl}
\dot{x}_k & = & -0.5\,x_k+x_k^3+ 0.05 u_k \,,  \\
y_k & = & x_k \,.
\end{array}
\end{equation*}
With the observation function $f(X)=x_1$ (one measure at one vertex), spectral network identification yields accurate estimations of the algebraic connectivity $\tilde{\lambda}_2=13.34$ and spectral radius $\tilde{\lambda}_n=42.98$ (Figure \ref{fig:example6} (top)). Using \eqref{d_min_max}, we can obtain a good approximation of the minimum and maximum vertex degrees: $d_{min}\geq 13.21$ (exact value: $d_{min}=14$) and $d_{max} \leq 42.55$ (exact value: $d_{max}=40$). Numerical simulations performed on $100$ different random graphs show that the bounds that we estimate on $d_{\min}$ and $d_{\max}$ are typically within $15\%$ of the actual values, see Table \ref{tab:example6}.

\begin{table}[h]
	\centering
			\begin{tabular}{lcc}
			\hline
		  & $d_{min}$ & $d_{max}$ \\
			\hline
			mean error (absolute) & 2.18 & 5.74 \\
			mean error (relative) & 0.11 & 0.14 \\
		root mean squared error (absolute) & 2.66 & 7.73 \\
		root mean squared error (relative) & 0.14 & 0.18 \\
		\hline
		\end{tabular}
		\caption{The minimum and maximum vertex degrees $d_{min}$ and $d_{max}$ are estimated for $100$ different random networks. The mean relative error is within $15\%$ of the actual value.}
		\label{tab:example6}
\end{table}

Now we add a new vertex that is connected to the rest of the network by only one edge. This new vertex is not directly connected to the measured vertex $1$, but is at a distance $3$ of it, i.e. the shortest path between these two vertices contains three edges. \alex{(Note that the average distance between two vertices of the graph is approximately $\ln 100/\ln 30 \approx 1.35$ \cite{Barabasi}.)} It follows from this modification that the minimum vertex degree drops to $d_{min}=1$ and it is remarkable that this can be detected by our estimate. Indeed, the (much smaller) value of the algebraic connectivity $\lambda_2=0.98$ is perfectly captured by the algorithm (Figure \ref{fig:example6} (bottom)), yielding the approximated minimum vertex degree $d_{min} \geq 0.97$. Similar results could be obtained when the graph is weighted.

\begin{figure}[htbp]
\centering
\subfigure{\includegraphics[width=2.5cm]{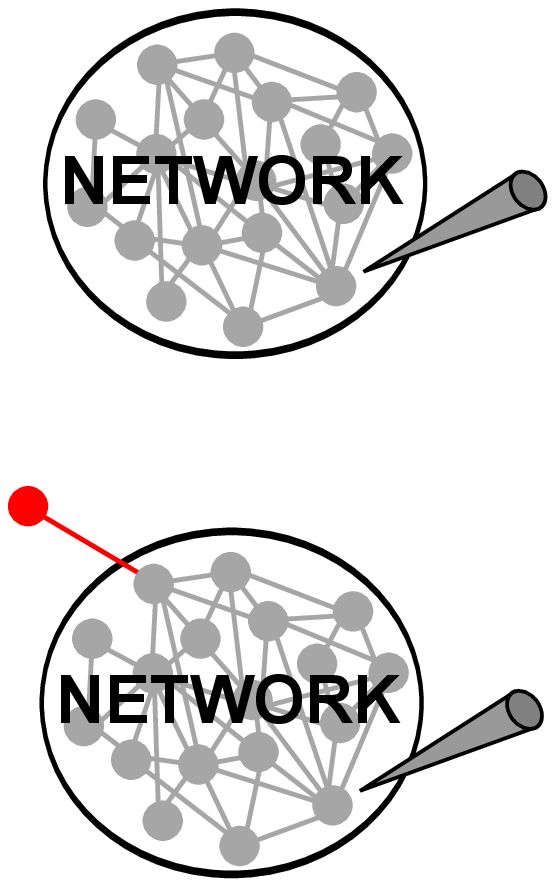}}
\subfigure{\includegraphics[height=5cm]{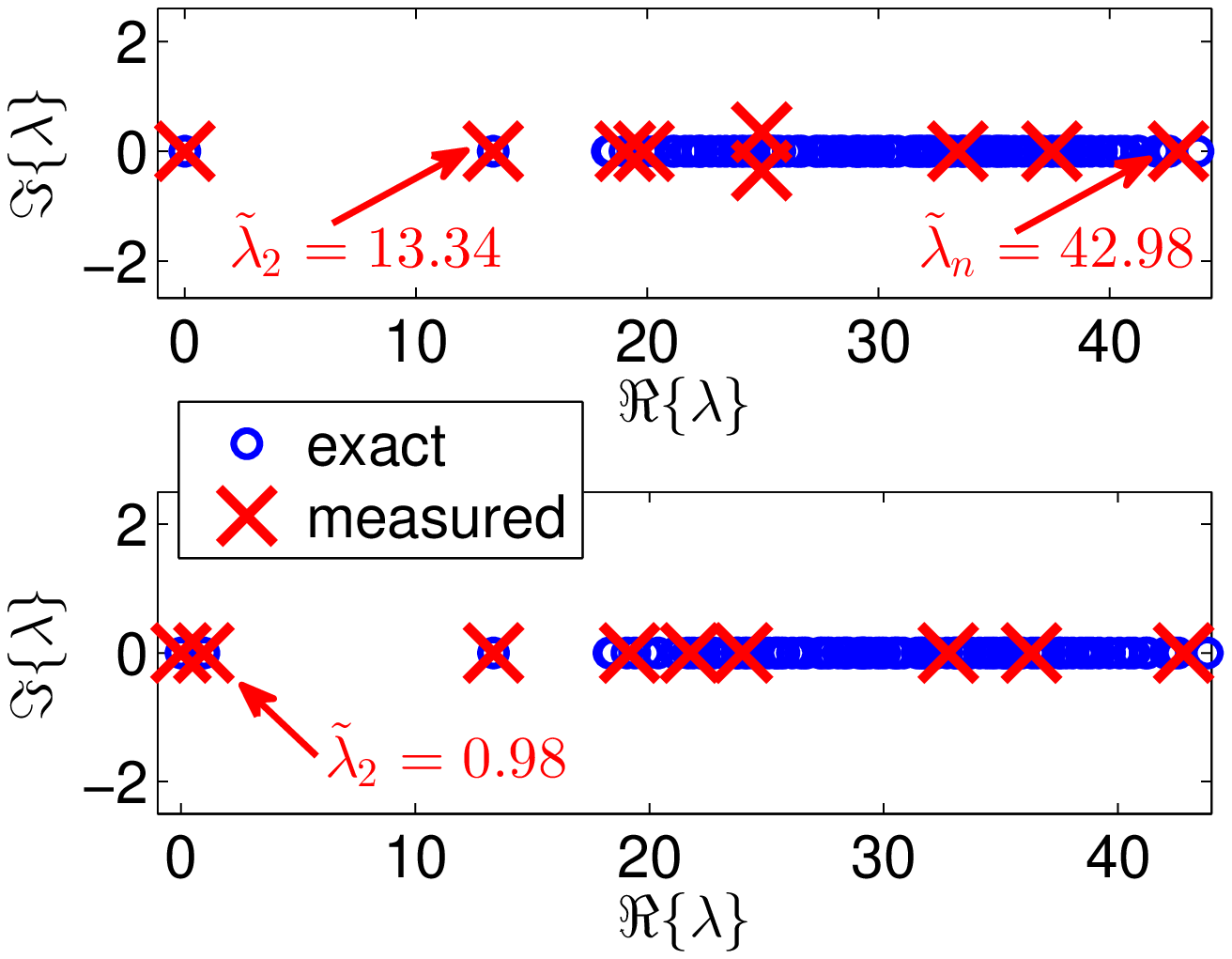}}
\caption{Identification of the algebraic connectivity $\lambda_2$ and spectral radius $\lambda_n$ of a (undirected) random graph. Simulation parameters are given in Appendix \ref{app_examples}. (Top) The approximations $\tilde{\lambda}_2$ and $\tilde{\lambda}_n$ provide a good estimation of the minimum and maximum vertex degrees. (Bottom) A vertex and one edge have been added to the network. The new (smaller) eigenvalue $\lambda_2$ is recovered and captures the minimum degree $d_{min}=1$.}
\label{fig:example6}
\end{figure}

\subsection{Average number of connections in a neural network}
\label{sec:example_Celegans}

We use spectral identification to estimate the average number of connections in a real neural network, through measures of a single neuron that has only one connection to the rest of the network.

We consider the neural network of C. Elegans \cite{Watts_Strogatz}, which is described by an undirected graph of $297$ vertices (i.e. neurons). The weights of the edges are randomly distributed according to a uniform distribution on $[0,0.02]$. The local dynamics of a neuron at vertex $k$ is governed by the FitzHugh-Nagumo equations \cite{FitzHugh,Nagumo}
\begin{equation*}
\begin{array}{ccl}
\mat{c}{\dot{V}_k \\ \dot{w}_k} & = & \mat{c}{-w_k-V_k (V_k-1)^2+1 \\ 0.5(V_k- w_k)}
+ \delta A_k \mat{c}{V_k \\ w_k} +\mat{c}{2 \\ 0} u_k \,, \\
y_k & = & 2 V_k \,,
\end{array}
\end{equation*}
where $V_k$ is the membrane voltage of the neuron, $w_k$ is the recovery variable, and $\delta A_k$ is a random matrix of mean $0$ and variance $s^2=0.1$ that accounts for the heterogeneity in the network.

We measure the membrane voltage variable of the neuron $k=54$ (i.e. $f(X)=V_{54}$) that is connected to only one other neuron, i.e. its unweighted degree is equal to one. Despite such an extremely local measure, we can obtain a fair estimation of the spectral Laplacian moments of the graph. The results are summarized in Figure \ref{fig:example_Celegans} and Table \ref{tab:example_Celegans}. \alex{The first Laplacian moment provides an estimation of the average number of connections $\mathcal{D}_1(\bar{\mathcal{G}})=11.15$ (exact value: $\mathcal{D}_1(\bar{\mathcal{G}})=7.9$). Although not negligible, the error is small compared with the standard deviation of the degree distribution, which is higher than $10$.} In contrast, a direct observation of the number of connections of the observed neuron would provide a poor estimation $\mathcal{D}_1(\bar{\mathcal{G}})=1$.

The second Laplacian moment is underestimated, so that \eqref{quad_mean_deg} cannot provide a good estimation of the quadratic mean number of connections $\mathcal{D}_2(\bar{\mathcal{G}})$. This is mainly due to the fact that the distribution of eigenvalues is not homogeneous (see Figure \ref{fig:example_Celegans}), so that the computation of the spectral moments with a convex hull of computed eigenvalues is not accurate. Other simulations not shown here can provide a better estimate of the second spectral moment, but in that case the estimation of the first moment is less accurate. This suggests that there is a tradeoff to obtain good approximations of the first and second moments. This issue \alex{could be explained by the nonlinearity of the model and the complexity of the network. It} could be tackled by more advanced numerical methods using for instance machine learning techniques to post-process the results.

\begin{figure}[htbp]
\centering
\includegraphics[width=8cm]{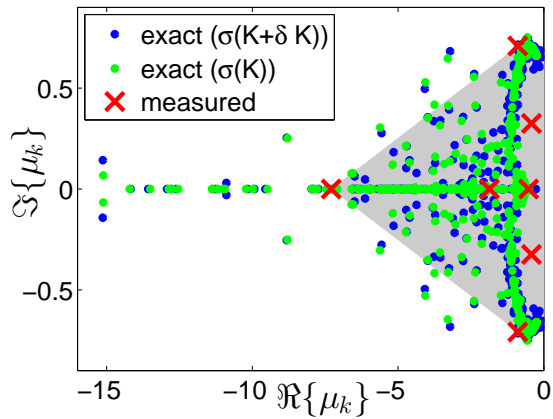}
\caption{Identification of the spectrum of $K+\delta K$ for the neural network of C. Elegans with nonlinear non-identical units. Simulation parameters are given in Appendix \ref{app_examples}. The convex hull (in gray) of the measured eigenvalues provides an approximation of the spectral moments of $K+\delta K$.}
\label{fig:example_Celegans}
\end{figure}

\begin{table}[h]
	\centering
		\begin{tabular}{ccc}
		\hline
		 & estimated & exact \\
		\hline
		$\mathcal{M}_1(K+\delta K)$ & -2.48 & -1.81 \\
		$\mathcal{M}_1(K)$ & -2.48 & -1.82 \\
		$\mathcal{M}_1(L)$, $\mathcal{D}_1(\mathcal{G})$ & 1.12 & 0.78 \\
		$\mathcal{M}_1(\bar{L})$ & 11.15 & 7.90 \\
		\hline
		\end{tabular}
		\hspace{1cm}
	  \begin{tabular}{ccc}
		\hline
		 & measured & exact \\
		\hline
		$\mathcal{M}_2(K+\delta K)$ & 8.98 & 13.76 \\
		$\mathcal{M}_2(K)$ & 8.98 & 13.76 \\
		$\mathcal{M}_2(L)$ & 1.17 & 1.77 \\
		$\mathcal{M}_2(\bar{L})$ & 113.03 & 170.50 \\
		\hline
		\end{tabular}
		\caption{Identification of the spectral moments for the neural network of C. Elegans with nonlinear non-identical units. The results show a good approximation of the first moment, providing a fair estimate of the mean number of connections in the network. However, the second spectral moment is underestimated.}
		\label{tab:example_Celegans}
\end{table}

\subsection{Mutual influence of two units}
\label{sec:example_karate}

Spectral network identification also allows to detect whether two vertices are part of the same connected network, and therefore dynamically influence each other.

Here we consider the Zachary's karate club network \cite{Zachary}, which describes friendships between $34$ people. The graph is unweighted, and we make it directed by randomly assigning a direction to each edge. The local dynamics at each vertex is a nonlinear consensus dynamics described by
\begin{equation*}
\dot{x}_k = 0.2 \tanh \left(\frac{1}{2} u_k \right) = 0.2 \tanh \left(\frac{1}{2} \sum_{j=1}^n W_{kj} (x_j-x_k) \right)
\end{equation*}
where the function $\tanh$ plays the role of a saturation. The state $x_k$ represents the opinion of unit $k$.

Measuring the opinion of $k=12$, we can obtain a good approximation of the dominant Laplacian eigenvalues of the graph. Another experiment using measurements of the opinion of $k=30$ yields similar results (Figure \ref{fig:example_karate}(a)), suggesting that units $12$ and $30$ belong to the same connected component of the network and have an influence on each other. Now the network is cut in two connected components by removing the edges $(1,	32)$, $(2,	31)$, $(3,	10)$, $(3,	28)$, $(3,	29)$, $(3,	33)$, $(9,	31)$, $(9,	33)$, $(9,	34)$, $(14,	34)$, and $(20,	34)$ (this corresponds to the actual split of the network described in \cite{Zachary}). Spectral identification performed on this disconnected network yields different results when the measurements are taken on unit $12$ or on unit $30$ (Figure \ref{fig:example_karate}(b)). In particular, Laplacian eigenvalues captured with measurements of one unit cannot be captured with measurements of the other unit. This indicates that the two units do not belong to the same connected component and have no influence on each other.
\begin{figure}[htbp]
\centering
\subfigure[Without cut]{\includegraphics[width=0.45\columnwidth]{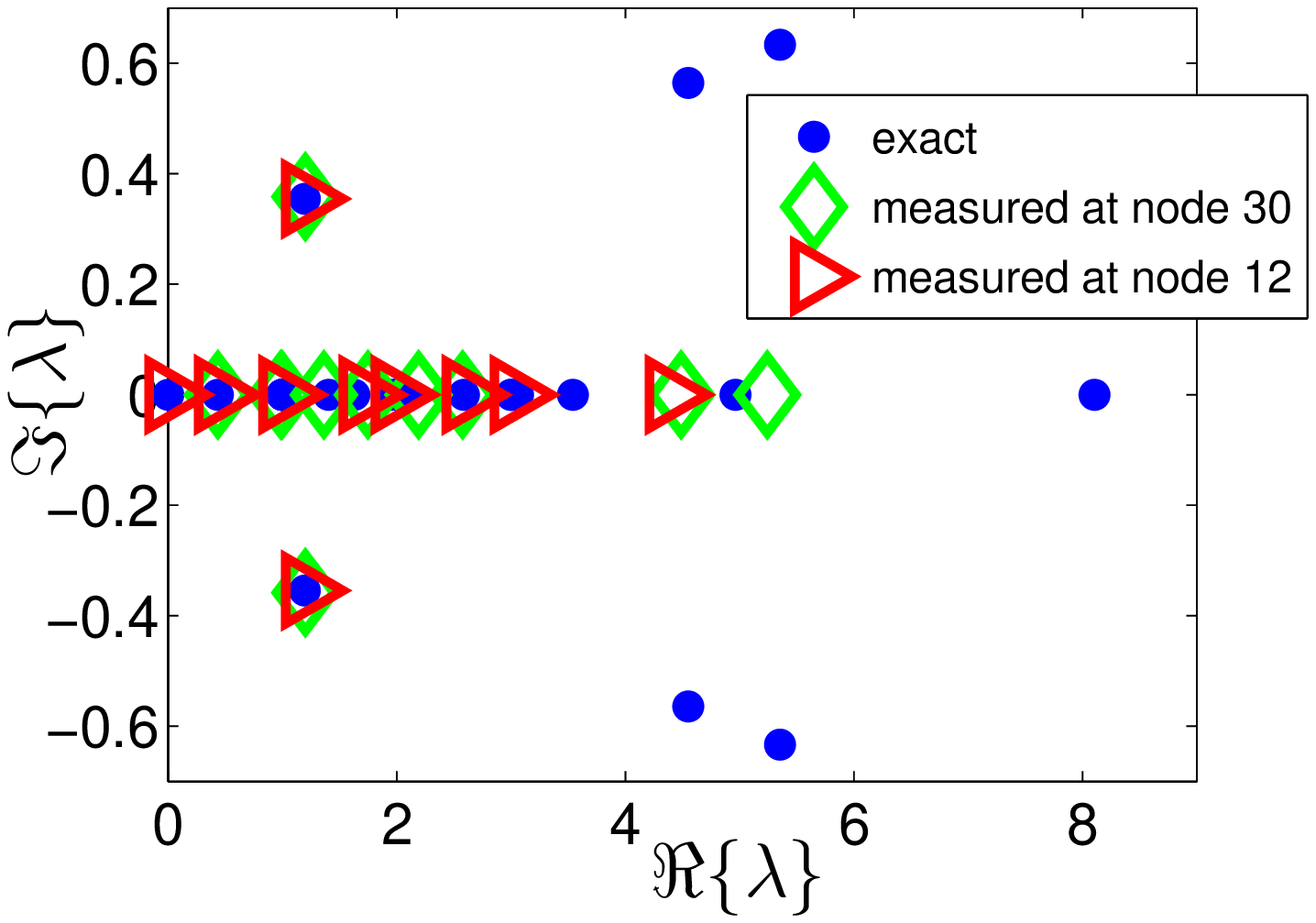}}
\subfigure[With cut]{\includegraphics[width=0.45\columnwidth]{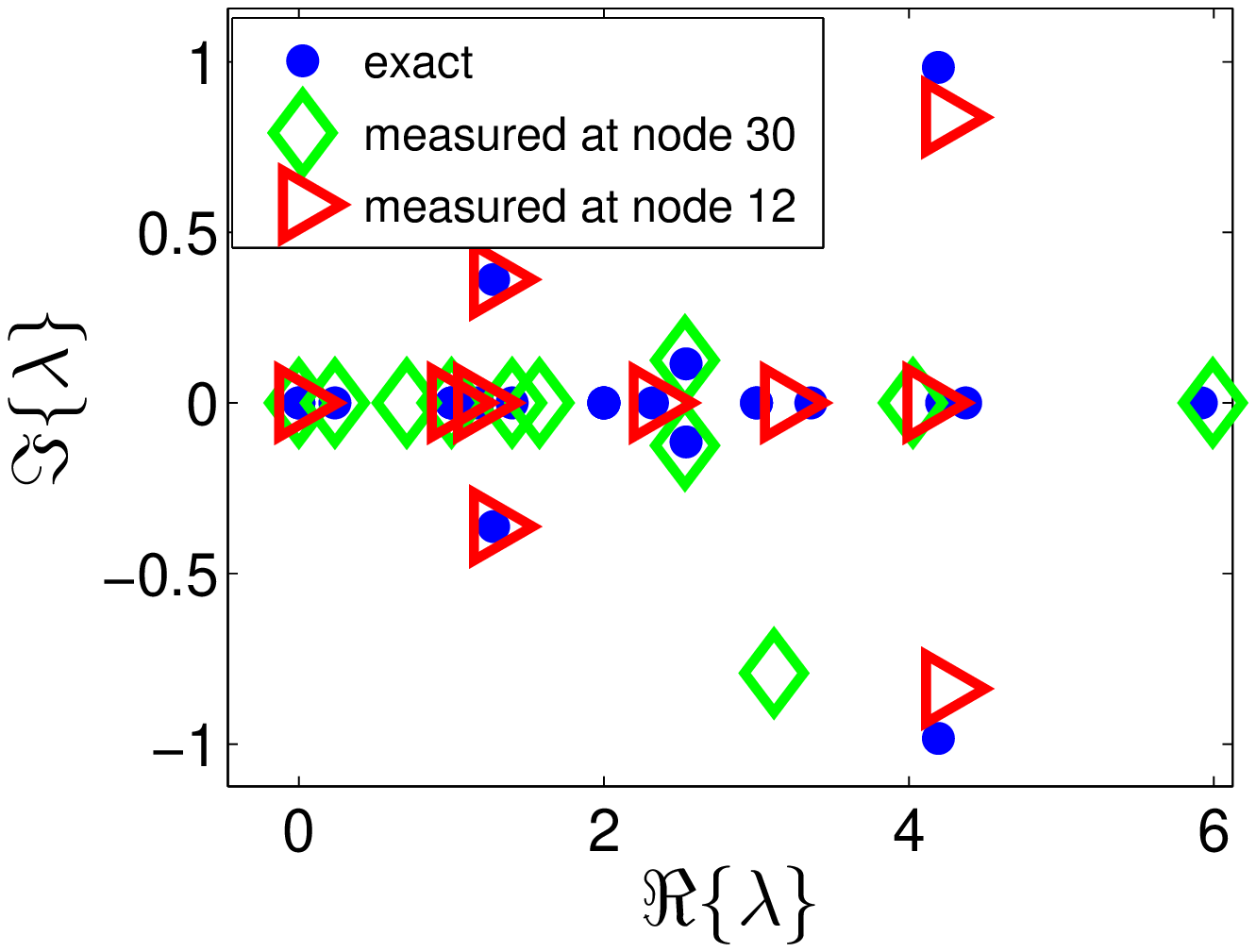}}
\caption{Spectral network identification of the karate club network with nonlinear consensus dynamics. Simulation parameters are given in Appendix \ref{app_examples}. (a) The network is connected, so that similar results are obtained with measurements of units $12$ or $30$. (b) The network is cut in two connected components. The sets of Laplacian eigenvalues obtained with the measurements at unit $12$ and unit $30$ are almost disjoint, showing that the two units do not belong to the same connected component of the network.}
\label{fig:example_karate}
\end{figure}

This example shows that vertices belonging to different connected components of a network capture different subsets of the Laplacian spectrum. Similarly, vertices belonging to different groups (or communities) in the network might also \guillemets{feel} different subsets of the Laplacian spectrum. This suggests that spectral network identification could be used for community detection and network partitioning; see e.g. \cite{Raak}.

\section{Conclusion}
\label{sec:conclu}

This paper is a first step toward the development of a novel framework for the identification of networks, which focuses on spectral graph-theoretic properties rather than full graph topology. Spectral network identification reveals global information on the network with only sparse local measurements, a key feature that is of great interest when dealing with large real networks. 

We have obtained very encouraging results. Numerical examples show for example that one can estimate the mean number of connections and the minimum/maximum vertex degree, with a better accuracy than through direct observations. Moreover, one can detect the addition of a vertex, and measure whether two units interact indirectly in the same connected component of the network.

Our main theoretical results highlight the connection between the spectral properties of the network and the spectral properties of the collective dynamics expressed within the framework of the Koopman operator. When the network is composed of non-identical units, we showed that it is not possible to derive such an exact connection. However, a statistical approach focusing on spectral moments can overcome this limitation in the case of large networks.

\paragraph{Perspectives} The present study has focused on networks of (nonlinear) units interacting through a diffusive coupling and admitting a stable synchronized equilibrium. Further research perspective is to extend the spectral identification framework to other dynamics that exhibit more complex or non-synchronized behaviors (e.g. coupled oscillators). \alex{The theoretical framework can be easily extended to synchronized limit-cycle oscillators, but the numerical methods presented in this paper are not effective in this case and should be improved. For agents that do not synchronize, additional research is necessary to tackle some limitations that are also encountered in the case of non-identical agents (i.e. impossibility results, see Section \ref{sec:impossibility}).} Moreover, considering nonlinear couplings would be particularly relevant to many applications. Note that with different types of coupling, one could also capture other spectral properties, such as the spectrum of the adjacency matrix. In fine, the spectral identification framework should be applied to the analysis of real data, providing spectral markers to classify different networks. It could also be compared to recent methods for detecting a pathology or a fault in the network \cite{fault_detection}.

The numerical method developed in this paper relies on the DMD algorithm, which extracts spectral properties from measured collective dynamics of the network. Recent extensions of the DMD algorithm (e.g. extended DMD \cite{Rowley_EDMD}, de-biased DMD \cite{debiased_DMD}, compressed DMD \cite{compressed_DMD}) or the related Arnoldi-based method (\cite{Yoshi_prony}) could be implemented in the context of spectral network identification. The DMD algorithm is primarily designed to capture the (Koopman) modes of the system, which are not needed in the present study. Future theoretical research could attempt to exploit these modes, as it is done for instance in \cite{Raak} for network partitioning. More importantly, the DMD algorithm is not always accurate and effective, \alex{in particular for strongly nonlinear systems or when initial conditions are far from the equilibrium.} The performances of the method, which might also depend on the network structure (e.g. graph diameter), should therefore be further investigated. \alex{In addition, a novel numerical scheme could be required, which is well-suited to strong nonlinearities and focuses on (Koopman) eigenvalues instead of (Koopman) modes, since the latter are the main quantities of interest.} This new scheme could also be combined with other techniques, such as machine learning, to post-process the obtained spectral data and obtain an upper bound on the error as well as a better estimate of the spectral moments. This could lead to advanced and automatic methods that would significantly increase the potential of spectral network identification.

\appendix

\section{Proofs}
\label{sec:app_proofs}

\subsection{Proof of the inequalities \eqref{quad_mean_deg}}
We have
\begin{equation}
\label{inequa_M2}
\mathcal{M}_2(L)= \frac{1}{n} \sum_{i,j} L_{ij} L_{ji} = \frac{1}{n} \sum_i d_i^2 + \frac{1}{n} \sum_{\substack{i,j \\ i \neq j}} W_{ij} W_{ji} \geq \frac{1}{n} \sum_i d_i^2 = \mathcal{D}_2(\mathcal{G})
\end{equation}
since $W_{ij} \geq 0$. Moreover, the inequality $2 W_{ij} W_{ji} \leq W_{ij}^2+W^2_{ji}$ for all $i\neq j$ implies that
\begin{equation*}
\sum_{\substack{i,j \\ i \neq j}} W_{ij} W_{ji} \leq \sum_{\substack{i,j \\ i \neq j}} W_{ij}^2 \leq \sum_i d_i^2
\end{equation*}
and it follows from \eqref{inequa_M2} that
\begin{equation*}
\mathcal{M}_2(L) \leq \frac{2}{n} \sum_i d_i^2
\end{equation*}
or equivalently $\mathcal{D}_2(\mathcal{G}) \geq \mathcal{M}_2(L)/2$. Finally, we have
\begin{equation*}
(\mathcal{M}_1(L))^2 = \frac{1}{n^2} \sum_i d_i^2 + \frac{1}{n^2} \sum_{\substack{i,j \\ i \neq j}} d_i d_j \leq \frac{1}{n^2} \sum_i d_i^2 + \frac{n-1}{n^2} \sum_i d_i^2 = \mathcal{D}_2(\mathcal{G})\,.
\end{equation*}

\subsection{Proof of Proposition \ref{prop_rand_delta_K}}
\paragraph{Expectation of $\widehat{\mathcal{M}_1}(L)$} We have
\begin{equation*}
\mathbb{E}(\tr(K+\delta K)) = \tr(K) + \mathbb{E}(\tr(\delta K)) = \tr(K)
\end{equation*}
since $\mathbb{E}(\tr(\delta K))=0$ (see \eqref{random_delta_K}), and equivalently
\begin{equation*}
\mathbb{E}(\mathcal{M}_1(K+\delta K)) = \mathcal{M}_1(K) \,.
\end{equation*}
Then it follows from \eqref{M1_L} that $E(\widehat{\mathcal{M}_1}(L)) = \mathcal{M}_1(L)$.\\

\paragraph{Expectation of $\widehat{\mathcal{M}_2}(L)$} It follows that
\begin{equation*}
\begin{split}
\mathbb{E}(\tr(K+\delta K)^2) & = \tr(K^2) + \mathbb{E}(\tr(\delta K^2)) + 2 \mathbb{E} ( \tr (K\, \delta K)) \\
& = \tr(K^2) + \mathbb{E} \Big (\sum_{i,j} \delta K_{ij} \, \delta K_{ji} \Big ) + 2 \mathbb{E} (\sum_{i,j} K_{ij} \, \delta K_{ji} \Big ) \\
& = \tr(K^2) + mn\, s^2
\end{split}
\end{equation*}
where the last equality follows from $\mathbb{E}(\tr(\delta K))=0$ and \eqref{random_delta_K}. Then we have
\begin{equation*}
\mathbb{E}(\mathcal{M}_2(K+\delta K)) - s^2 = \mathcal{M}_2(K)
\end{equation*}
and it follows from \eqref{M2_L} that $E(\widehat{\mathcal{M}_2}(L)) = \mathcal{M}_2(L)$.

\paragraph{Variance of $\widehat{\mathcal{M}_1}(L)$} We obtain
\begin{equation*}
\var(\tr(K+\delta K)) = \mathbb{E}(\tr(\delta K))^2 = \mathbb{E} \Big (\sum_{i,j} \delta K_{ii} \, \delta K_{jj} \Big) = mn \, s^2
\end{equation*}
or equivalently
\begin{equation*}
\var(\mathcal{M}_1(K+\delta K)) = \frac{s^2}{mn} \,.
\end{equation*}
Then it follows that
\begin{equation*}
\var(\widehat{\mathcal{M}_1}(L)) = \frac{m^2}{(C^TB)^2} \var(\mathcal{M}_1(K+\delta K)) = \frac{m s^2}{n (C^TB)^2}\,.
\end{equation*}

\paragraph{Variance of $\widehat{\mathcal{M}_2}(L)$}
It follows from \eqref{M1_estim} and \eqref{M2_estim} that
\begin{equation}
\label{terms_variance}
\begin{split}
\var(\widehat{\mathcal{M}_2}(L)) = \mathcal{O} \Big(& \var(\mathcal{M}_1(K+\delta K)) + \var(\mathcal{M}_2(K+\delta K)) \\
& \qquad \qquad +  \cov(\mathcal{M}_1(K+\delta K),\mathcal{M}_2(K+\delta K)) \Big) \,.
\end{split}
\end{equation}
We consider each term of \eqref{terms_variance} separately.\\\\
\emph{1. Term $\var(\mathcal{M}_1(K+\delta K))$ in \eqref{terms_variance}.} We have shown that
\begin{equation}
\label{scale_var1}
\var(\mathcal{M}_1(K+\delta K))=\mathcal{O}\left(\frac{1}{n}\right) \,.
\end{equation}
\emph{2. Term $\var(\mathcal{M}_2(K+\delta K))$ in \eqref{terms_variance}.} We have
\begin{equation*}
\begin{split}
\var(\tr(K+\delta K)^2) & = \mathbb{E}\left((\tr(K +\delta K)^2)^2\right) - \left(\mathbb{E}(\tr(K+\delta K)^2)\right)^2 \\
& = \mathbb{E}\left(\tr(K^2) + \tr(\delta K^2) + 2 \tr(K \, \delta K)\right)^2 \\
& \qquad - \left(\tr(K^2) + \mathbb{E}(\tr(\delta K^2)) + 2 \mathbb{E}(\tr(K \, \delta K))\right)^2 \\
& = \mathbb{E}(\tr(\delta K^2))^2 + 4 \, \mathbb{E}(\tr(K \, \delta K))^2 + 4 \mathbb{E}(\tr(\delta K^2) \, \tr(K\, \delta K)) \\
& \qquad - (\mathbb{E}(\tr(\delta K^2)))^2 - 4\, (\mathbb{E}(\tr(K \, \delta K)))^2 - 4 \, \mathbb{E}(\tr(\delta K^2)) \, \mathbb{E}(\tr(K \, \delta K)) \\
& = \mathbb{E} \Big(\sum_{i,j} \sum_{i',j'} \delta K_{ij} \, \delta K_{ji} \, \delta K_{i'j'} \, \delta K_{j'i'} \Big) + 4 \, \mathbb{E} \Big(\sum_{i,j} \sum_{i',j'} K_{ij} \, \delta K_{ji} \, K_{i'j'} \, \delta K_{j'i'} \Big) \\
& \qquad + 4 \, \mathbb{E} \Big(\sum_{i,j} \sum_{i',j'} \delta K_{ij} \, \delta K_{ji} \, K_{i'j'} \, \delta K_{j'i'} \Big) - n^2 m^2\, s^4 \\
& = \sum_i \mathbb{E}(\delta K_{ii}^4) + \sum_{\substack{i,i' \\ i \neq i'}} \mathbb{E}(\delta K_{ii}^2) \, \mathbb{E}(K_{i'i'}^2) + 2 \sum_{\substack{i,j \\ i\neq j}} \mathbb{E}(\delta K_{ij}^2) \, \mathbb{E}(\delta K_{ji}^2) \\
& \qquad +  4 \sum_{i,j} K_{ij}^2  \, \mathbb{E}(\delta K_{ji}^2) + 4 \sum_i K_{ii} \, \mathbb{E}(\delta K_{ii}^3) - n^2 m^2\, s^4 \\
& = mn \, \mathbb{E}(\delta K_{ii}^4) + mn (mn-1) \, s^4 + 2 mn(m-1) s^4 +4 s^2 \sum_{\substack{i,j \\ \delta K_{ij} \neq 0}} K_{ij}^2  \\
& \qquad + 4 \mathbb{E}(\delta K_{ii}^3) \, \tr(K) - n^2 m^2\, s^4
\end{split}
\end{equation*}
where we used $\mathbb{E}(\tr(\delta K))=0$, $\mathbb{E}(\tr(K\, \delta K))=0$, and the block-diagonal structure of $\delta K$. It follows that
\begin{equation*}
\begin{split}
\var(\mathcal{M}_2(K+\delta K)) & = \frac{\mathbb{E}(\delta K_{ii}^4)}{mn} +  s^4 \frac{m-2}{mn} + s^2 \frac{4}{m^2n^2}  \sum_{\substack{i,j \\ \delta K_{ij} \neq 0}} K_{ij}^2  + 4 \mathbb{E}(\delta K_{ii}^3) \frac{\tr(K)}{m^2 n^2} \\
& = \mathcal{O} \left(\frac{1}{n}+\frac{1}{n^2} \sum_{\substack{i,j \\ \delta K_{ij} \neq 0}} K_{ij}^2 + \frac{\mathcal{D}_1(\mathcal{G})}{n} \right)
\end{split}
\end{equation*}
where we used $\tr(K)/n=\mathcal{M}_1(K)=\mathcal{O}(\mathcal{M}_1(L))=\mathcal{O}(\mathcal{D}_1(\mathcal{G}))$. It is also easy to see that
\begin{equation*}
\begin{split}
\frac{1}{n^2} \sum_{\substack{i,j \\ \delta K_{ij} \neq 0}} K_{ij}^2 & = \frac{\tr(AA^T)}{n} + \frac{1}{n^2} \left(\sum_{i} L_{ii}^2 \right) \tr(BC^TCB^T) - \frac{2}{n^2} \tr(L) \tr(A^T B C^T) \\
& = \mathcal{O} \left(\frac{1}{n}+ \frac{\mathcal{D}_2(\mathcal{G})}{n} + \frac{\mathcal{D}_1(\mathcal{G})}{n} \right)
\end{split}
\end{equation*}
so that
\begin{equation}
\label{scale_var2}
\var(\mathcal{M}_2(K+\delta K)) = \mathcal{O} \left(\frac{1}{n}+ \frac{\mathcal{D}_1(\mathcal{G})}{n} + \frac{\mathcal{D}_2(\mathcal{G})}{n} \right)\,.
\end{equation}
\emph{3. Term $\cov(\mathcal{M}_1(K+\delta K),\mathcal{M}_2(K+\delta K))$ in \eqref{terms_variance}.} We have
\begin{equation*}
\begin{split}
\cov(\tr(K+\delta K),\tr(K+\delta K)^2) & = \mathbb{E}(\tr(K+\delta K) \tr(K+\delta K)^2) - \mathbb{E}(\tr(K+\delta K)) \mathbb{E}(\tr(K+\delta K)^2) \\
& = \mathbb{E}(\tr(\delta K) \tr(K+\delta K)^2) \\
& = \mathbb{E}(\tr(\delta K) \tr(\delta K^2)) + 2 \mathbb{E}(\tr(\delta K) \tr(K \delta K)) \\
& = \mathbb{E} \left( \sum_{i,i',j'} \delta K_{ii} \delta K_{i'j'} \delta K_{j'i'} \right) + 2 \mathbb{E} \left(\sum_{i,i',j'} \delta K_{ii} K_{i'j'} \delta K_{j'i'} \right) \\
& = mn \, \mathbb{E}(\delta K_{ii}^3) + 2 s^2 \tr(K)
\end{split}
\end{equation*}
or equivalently
\begin{equation}
\label{scale_var3}
\cov(\mathcal{M}_1(K+\delta K),\mathcal{M}_2(K+\delta K)) = \frac{\mathbb{E}(\delta K_{ii}^3)}{mn} + 2 s^2 \frac{\mathcal{M}_1(K)}{mn} = \mathcal{O} \left(\frac{1}{n} + \frac{\mathcal{D}_1(\mathcal{G})}{n} \right) \,.
\end{equation}

The result follows from \eqref{scale_var1}-\eqref{scale_var2}-\eqref{scale_var3}.

\subsection{Proof of Proposition \ref{prop:rand_weights}}
We first note the relationship between the weighted Laplacian matrix $L$ and the unweighted Laplacian matrix $\bar{L}$:
\begin{equation}
\label{unweighted_L}
\begin{array}{rcll}
L_{ij} & = & \bar{L}_{ij} W_{ij} & i \neq j \\
L_{ii} & = &  - \displaystyle \sum_{j \neq i} \bar{L}_{ij} W_{ij}  &
\end{array}\,.
\end{equation}

\paragraph{First spectral moment} It follows from \eqref{unweighted_L} that
\begin{equation*}
\mathbb{E}(\tr(L))=\mathbb{E}\Big(-\sum_{\substack{i,j \\ i \neq j}} \bar{L}_{ij} W_{ij} \Big) = - r \sum_{\substack{i,j \\ i \neq j}} \bar{L}_{ij} = r \, \tr(\bar{L})
\end{equation*}
or equivalently
\begin{equation}
\label{mom_unw1}
\mathbb{E}(\mathcal{M}_1(L))= r \, \mathcal{M}_1(\bar{L})\,.
\end{equation}

We have also
\begin{equation*}
\begin{split}
\var(\tr(L)) & = \mathbb{E}(\tr(L))^2 - \left( \mathbb{E}(\tr(L)) \right)^2 \\
& = \mathbb{E} \Big( \sum_{\substack{i,j \\ i \neq j}} \sum_{\substack{i',j' \\ i' \neq j'}} \bar{L}_{ij} \bar{L}_{i'j'} W_{ij}  W_{i'j'}   \Big) - r^2 \, (\tr(\bar{L}))^2 \\
& = s^2 \sum_{\substack{i,j \\ i \neq j}} \bar{L}_{ij}^2 + r^2 \sum_{\substack{i,j \\ i \neq j}} \sum_{\substack{i',j' \\ i' \neq j'}} \bar{L}_{ij} \bar{L}_{i'j'} - r^2 \, (\tr(\bar{L}))^2 \\
& = s^2 \tr(\bar{L})
\end{split}
\end{equation*}
where we used $\sum_{\substack{i,j \\ i \neq j}} \bar{L}_{ij}^2 = -\sum_{\substack{i,j \\ i \neq j}} \bar{L}_{ij} = \tr(\bar{L})$ and $\sum_{\substack{i,j \\ i \neq j}} \sum_{\substack{i',j' \\ i' \neq j'}} \bar{L}_{ij} \bar{L}_{i'j'} = (\tr(\bar{L}))^2$.\\
It follows that
\begin{equation*}
\var{} \left(\frac{\mathcal{M}_1(L)}{r}\right) = \frac{s^2}{r^2} \frac{\tr(\bar{L})}{n^2} = \frac{s^2}{r^2} \frac{\mathcal{D}_1(\bar{\mathcal{G}})}{n} < \frac{s^2}{r^2} \,.
\end{equation*}

\paragraph{Second spectral moment} We have
\begin{equation*}
\begin{split}
\mathbb{E}(\tr(L^2)) & = \mathbb{E} \Big(\sum_{\substack{i,j \\ i \neq j}} L_{ij} L_{ji} + \sum_i L_{ii}^2 \Big)\\
& = \mathbb{E} \Big(\sum_{\substack{i,j \\ i \neq j}} \bar{L}_{ij} \bar{L}_{ji} W_{ij} W_{ji} + \sum_i \sum_{j\neq i} \sum_{j'\neq i} \bar{L}_{ij} \bar{L}_{ij'} W_{ij} W_{ij'} \Big) \\
& = r^2 \sum_{\substack{i,j \\ i \neq j}} \bar{L}_{ij} \bar{L}_{ji} + r^2 \sum_i \sum_{j\neq i} \sum_{j'\neq i} \bar{L}_{ij} \bar{L}_{ij'} + s^2 \sum_{\substack{i,j \\ i \neq j}} \bar{L}_{ij}^2 \\
& = r^2 \, \tr(\bar{L}^2) + s^2 \tr(\bar{L})\,
\end{split}
\end{equation*}
or equivalently
\begin{equation}
\label{mom_unw2}
\mathbb{E}(\mathcal{M}_2(L)) = r^2 \, \mathcal{M}_2(\bar{L}) + s^2 \, \mathcal{M}_1(\bar{L}) \,.
\end{equation}
Using \eqref{mom_unw1} and \eqref{mom_unw2}, we get
\begin{equation*}
\mathbb{E}\left(\frac{\mathcal{M}_2(L)}{r^2}-\frac{s^2 \, \mathcal{M}_1(L)}{r^3}\right) = \mathcal{M}_2(\bar{L}) \,.
\end{equation*}

\section{Algorithms}
\label{app_algo}
The general procedure for spectral network identification is summarized in Algorithm \ref{alg:spectral_ident}. The DMD algorithm is described in Algorithm \ref{alg:dmd} (see also \cite{Schmid,Tu} for more details).

\begin{algorithm}[h]
	\caption{Spectral network identification}
	\label{alg:spectral_ident}
	\begin{algorithmic}[1]
\State Choose an observation function $f$ and measure $r$ time series $f(X^{(j)}(t))$;
\State Choose the parameters $K$, $\Delta t$ and obtain the snapshots \eqref{time_series} from the times series;
\State Choose the shift parameters $c$ and $\delta$ and construct the data matrix $Z$;
\State Apply the DMD algorithm to the data matrix $Z$ and obtain the DMD eigenvalues $\tilde{\nu}_k$ and the associated eigenvalues $\tilde{\mu}_k=\log(\nu_k)/\Delta t$ (see Appendix \ref{app_algo});
\State Optional: remove outliers (characterized by large real part or large imaginary part);
\State Optional: if the units are identical: use \eqref{sigma_L}-\eqref{g} to obtain exact Laplacian eigenvalue(s) (e.g. spectral gap $\lambda_2$ and spectral radius $\lambda_n$);
\State Use k-means clustering to identify $n_c$ clusters of eigenvalues $\tilde{\mu}_k$ and remove outliers (optional);
\State Compute the convex hull $\mathcal{S}_j$ of each cluster and the moments of area \eqref{moments_area};
\State Compute the spectral moments of $K$ using \eqref{area_moment_K};
\State If the units are identical, compute the spectral moments of $L$ using \eqref{M1_L} and \eqref{M2_L}; if the units are non-identical, use \eqref{M1_estim} and \eqref{M2_estim};
\State If the distributions of the edges weight is known, compute the spectral moments of $\bar{L}$ using \eqref{M1_unw}-\eqref{M2_unw}.
	\end{algorithmic}
\end{algorithm}

\begin{algorithm}[h]
	\caption{Dynamic Mode Decomposition}
	\label{alg:dmd}
	\begin{algorithmic}[1]
\Statex{\bf Input:} Data matrix $Z=[z_0 \cdots z_{q_2}] \in \mathbb{R}^{q_1 \times (q_2+1)}$.
\Statex{\bf Output:} DMD modes $\phi_k$ and associated DMD eigenvalues $\tilde{\nu}_k$.
\State
Construct the matrices
\begin{equation*}
X=[z_0 \cdots z_{q_2-1}]\in \mathbb{R}^{q_1 \times q_2} \qquad Y=[z_1 \cdots z_{q_2}]\in \mathbb{R}^{q_1 \times q_2}\,;
\end{equation*}
\State Compute the reduced singular value decomposition of $X$, i.e. $X = U \Sigma V^*$;
\State Construct the matrix $\tilde{T} = U^* Y V \Sigma^{-1}$;
\State Compute the eigenvectors $w_k$ and eigenvalues $\tilde{\nu}_k$ of $\tilde{T}$, i.e. $\tilde{T} w_k = \tilde{\nu}_k w_k$;
\State The DMD modes are given by $\phi_k = U w_k$ and the associated DMD eigenvalues are $\tilde{\nu}_k$.
	\end{algorithmic}
\end{algorithm}

Note that, instead of using Algorithm \ref{alg:dmd}, one could directly consider the spectral decomposition of $T = Y X^\dagger$, where $X^\dagger$ denotes the Moore-Penrose pseudoinverse of $X$. However, this procedure is less efficient that the standard DMD algorithm (Algorithm \ref{alg:dmd}).

\section{Simulation parameters}
\label{app_examples}

The input matrix used with the DMD algorithm depends on several parameters. The parameters used for the numerical experiments shown in this paper are given in the following table.

\begin{table}[h]
	\centering
		\begin{tabular}{ccccccc}
		\hline
   & distribution IC & \# IC & \# snapshots & sampling period & shift \\
	& & $r$ & $K$ &  $\Delta t$ & $c,\,\delta$ \\
	\hline
	Example \ref{example1} & normal $N(0,1)$ & 10 & 50 & 0.4 & 2, 5 \\
	Example \ref{example2} & uniform on $[-0.5,0.5]$ & 10 & 50 & 0.8 & 3, 5 \\
	Examples \ref{example3},\,\ref{example5} & normal $N(0,1)$ & 10 & 50 & 0.4 & 2, 5 \\
	Example \ref{example4} & normal $N(0,1)$ & 10 & 50 & 0.2 & 2, 5 \\
	Section \ref{sec:example6} & uniform on $[-0.5,0.5]$ & 5/10 & 50/40 & 1/0.6 & 2, 5 \\
	Section \ref{sec:example_Celegans} & uniform on $[-0.5,0.5]$ & 10 & 75 & 0.2 & 2, 5 \\
	Section \ref{sec:example_karate} & uniform on $[-1,1]$ & 20 & 100 & 1 & 2, 5 \\
	\hline
		\end{tabular}
		\caption{}
\end{table}

Adjacency matrix of the network considered in Example \ref{example1} and Example \ref{example2}:

\begin{equation*}
W=\mat{cccccccccc}{
         0  &  0.1466  &  0.0075  &  0.0238    &    0  &  0.1048  &  0.1446  &  0.1913    &     0  &  0.1758 \\
    0.1780  &  0 &   0.1909  &  0.1436    &     0    &     0  &  0.1696   &      0  &  0.1865  &  0.0153 \\
    0.1470  &  0.1313  &  0  &  0.0175  &  0.0575   & 0.0651    &     0  &  0.0009  &  0.1297 &  0.1010 \\
    0.1021  &  0.1326  &  0.1169  &  0   & 0.0374   & 0.1238    &     0     &    0  &  0.0061 &   0.0342 \\
    0.0929  &  0.1052   &      0  &  0.0200  & 0    &     0   & 0.1646   &      0    &     0  &  0.0182 \\
    0.0058  &       0    &     0   & 0.0795  &  0.0557  &  0  &       0  &  0.0892  &  0.1143  &  0.0165 \\
    0.0850  &  0.1551   & 0.1060   & 0.1301   &      0  &  0.0149  &       0  &  0.1748  &  0.1631  &  0.0509 \\
    0.0524  &       0   &      0   & 0.1887   & 0.0079  &  0.1701  &       0  &       0  &  0.0632   & 0.1182 \\
         0  &  0.1926   & 0.1678   & 0.0500   & 0.1875  &       0  &  0.1289  &       0   &      0  &  0.0557 \\
         0  &  0.0028   & 0.1156   &      0   & 0.0236  &  0.0635  &       0  &  0.0228   &      0  &  0 }
\end{equation*}
 

\section*{Acknowledgments}

The authors thank S. Kolumban for his insightful comments and helpful discussion on the manuscript.

\bibliographystyle{siamplain}

\end{document}